%% file: IntrMirrKP2.tex
\documentclass[11pt]{amsart}
\usepackage{amscd}
\usepackage{amssymb}
\usepackage[centering,text={15.5cm,23cm}]{geometry} 
\usepackage[dvips,final]{graphicx}
\usepackage[dvipsnames]{xcolor}
\usepackage{caption}
\usepackage{comment}
\usepackage{wrapfig}
\usepackage[all]{xy}
\usepackage{mathrsfs}
\usepackage{mathtools}
\usepackage{marvosym}
\usepackage{stmaryrd}
\usepackage{textcase}
\usepackage{srcltx}
\usepackage{upgreek} 
\usepackage[inline]{enumitem} 
\usepackage{hyperref}
\hypersetup{
  colorlinks   = true,          
  urlcolor     = violet,          
  linkcolor    = blue,          
  citecolor   = violet             
}
\usepackage{tikz}
\usetikzlibrary{calc,intersections,patterns.meta,positioning,patterns}

\usepackage{booktabs}

\definecolor{shadecolor}{rgb}{1,0.9,0.7}

\newtheorem{theorem}{Theorem}[section]
\newtheorem{lemma}[theorem]{Lemma}
\newtheorem{lemma-definition}[theorem]{Lemma-Definition}
\newtheorem{proposition}[theorem]{Proposition}
\newtheorem{corollary}[theorem]{Corollary}

\newtheorem{maintheorem}{Theorem}

\theoremstyle{definition}

\newtheorem{assumption}[theorem]{Assumption}

\newtheorem{conventions}[theorem]{Conventions}

\newtheorem*{acknowledgement}{Acknowledgment}

\theoremstyle{remark}
\newtheorem{remark}[theorem]{Remark}

\numberwithin{equation}{section}
\numberwithin{figure}{section}

\newcommand {\lfor} {\llbracket}
\newcommand {\rfor} {\rrbracket}

\newcommand{\NN} {\mathbb{N}}
\newcommand{\ZZ} {\mathbb{Z}}
\newcommand{\QQ} {\mathbb{Q}}
\newcommand{\RR} {\mathbb{R}}
\newcommand{\CC} {\mathbb{C}}
\newcommand{\EE} {\mathbb{E}}
\newcommand{\FF} {\mathbb{F}}
\newcommand{\PP} {\mathbb{P}}
\renewcommand{\AA} {\mathbb{A}}
\renewcommand{\SS} {\mathbb{S}}
\newcommand{\GG} {\mathbb{G}}

\newcommand {\cC}  {\mathcal{C}}

\newcommand {\cL}  {\mathcal{L}}

\newcommand {\cR}  {\mathcal{R}}

\newcommand {\cT}  {\mathcal{T}}

\newcommand {\cX}  {\mathcal{X}}

\newcommand {\foF}  {\mathfrak{F}}

\newcommand {\foU}  {\mathfrak{U}}

\newcommand\bA{\mathbf{A}}

\newcommand\bsigma{{\boldsymbol{\sigma}}}

\newcommand\bdelta{{\boldsymbol{\delta}}}

\newcommand\btau{{\boldsymbol{\tau}}}
\newcommand\bomega{{\boldsymbol{\omega}}}
\newcommand\be{{\boldsymbol{e}}}

\newcommand {\Aut}  {\operatorname{Aut}}

\newcommand {\coker} {\operatorname{coker}}

\newcommand {\dlog} {\operatorname{dlog}}

\newcommand {\ev}  {{\operatorname{ev}}}

\newcommand {\sE} {{\mathsf E}}

\newcommand {\scrF}  {\mathscr{F}}
\newcommand {\fix}  {\mathrm{fix}}
\newcommand {\frakM} {\mathfrak{M}}

\newcommand {\gp}  {{\operatorname{gp}}}

\newcommand {\Int}  {\operatorname{Int}}

\renewcommand {\ker } {\operatorname{ker}}

\newcommand{\LL}{\mathbb{L}}

\newcommand {\lra}  {\longrightarrow}

\newcommand {\M} {\mathcal{M}}

\newcommand {\NE}  {\operatorname{NE}}

\renewcommand{\O}  {\mathcal{O}}

\newcommand {\ol} {\overline}

\newcommand {\scrM}  {\mathscr{M}}

\newcommand {\sing} {\mathrm{sing}}

\newcommand {\Spec} {\operatorname{Spec}}

\newcommand {\Tot}  {\operatorname{Tot}}

\newcommand {\virt}  {{\operatorname{virt}}}

\def\mydate{\ifcase\month \or January\or February\or March\or
April\or May\or June\or July\or August\or September\or October\or 
November\or December\fi \space\number\day,\space\number\year}

\date{\today}

\begin{document}

\title[Closed mirror symmetry from wall structures]{
From wall structures to closed mirror symmetry\\[.7ex]
\normalfont\large\NoCaseChange{The case of $K_{\PP^2}$: Renormalized periods over 
the positive real locus, closed Gromov--Witten invariants from wall functions, 
and tropical enumeration}}

\author{Michel van Garrel}
\author{Bernd Siebert}

\setcounter{tocdepth}{2}

\begin{abstract}
We recover closed Gromov--Witten invariants and renormalized mirror periods for
$K_{\PP^2}$ by the same operation on a wall function. This gives a direct
passage from the Gross--Siebert construction of intrinsic mirror pairs to
classical enumerative mirror symmetry. The link is a polynomiality theorem for
punctured invariants. Assuming the expected identification with the normalized
slab function, we also obtain a finite tree sum for closed Gromov--Witten
invariants in terms of types of plane tropical curves. The mechanism is expected
to extend to more general Calabi--Yau mirror pairs.
\end{abstract}

\maketitle
\tableofcontents


\section{Introduction}

Mirror symmetry as a mathematical research topic arguably began with the
prediction of the numbers of rational curves on quintic threefolds via
variations of period integrals in the work of Candelas, de la Ossa, Green,
Parkes \cite{CdGP91}. A big early success of the field was the proof of this
conjecture using localization by Givental \cite{Givental} and Lian, Liu,
Yau \cite{LianLiuYau}.

The generalization of these results to the mirror geometries provided by the
joint program of Mark Gross with the second author, and notably in terms of the
wall structure from \cite{affinecomplex}, has been a challenge for almost two
decades. The enumerative interpretation of scattering \cite{GPS}, tree
expansions \cite{ICM}, canonical coordinates \cite{RS}, the intrinsic mirror
construction via punctured Gromov--Witten invariants \cite{associativity,
CanScat}, and logarithmic mirror symmetry for $(\PP^2,E)$ \cite{vGRS1} have
provided increasingly concrete evidence that enumerative mirror symmetry should
be a tautology from this point of view.

The central challenge that has eluded us for so long is the treatment of a
planar slab function, with the minimal example the familiar mirror family
$\check X_t$ of (a toric degeneration of) $K_{\PP^2}$ reviewed in \S\ref{Sec:
intrinsic mirror geometry}:
\begin{enumerate}
\item[(A)]
How does one express the closed degree $d$ Gromov--Witten invariants
$N_d$ of $K_{\PP^2}$ in terms of the slab function that produces $\check X_t$?
\item[(B)]
How does an appropriate period integral that generalizes to compact
Calabi--Yau geometries compute the same expression in the slab function?
\item[(C)]
Do the $N_d$ have a tree sum expression analogous to the expansion of
the slab function as a sum over types of genus zero tropical maps to $\RR^2$ as
in \cite{ICM}?
\end{enumerate}
This paper answers all three questions transparently and, in our opinion, quite
beautifully. The method is expected to generalize at least to all
three-dimensional mirror pairs from \cite{affinecomplex}, and possibly to all
dimensions and even to intrinsic mirror pairs of \cite{CanScat}. We now state
our main results.

Write $f$ for the canonical slab function defined in \cite[\S3.2]{CanScat} in
the variables $x,y,z,t$ with $xyz=t$. We expect that $f=1+x+y+z+g(t)$ with
$g(t)\in t\cdot\CC\lfor t\rfor$ determined uniquely by the normalization
condition \cite{ICM}, but this is not needed for our answers to questions (A)
and (B). For $d,k>0$, put
\begin{equation}
\label{Eqn: W_{d,k}}
  W_{d,k}
    =\frac1k\left([t^dx^k]+[t^dy^k]+[t^dz^k]\right)\log f,
\end{equation}
where coefficients are read in the monomial expansion of \S\ref{SS: wall
invariants}. These numbers are sums of $1$-punctured Gromov--Witten invariants
of wall type in the three distinguished directions, namely
$W_{d,k}=\sum_{i=1}^3W_{d,-k\be_i}$ in the notation used there.

\begin{maintheorem}
\label{Thm: intro wall closed comparison}
For every $d>0$, the function
\[
  k\longmapsto(-1)^kW_{d,k},\qquad k\in\ZZ_{>0},
\]
extends to a polynomial of degree at most $2d-2$, whose value at $k=0$ is
$dN_d$. In particular, the closed invariants are recovered from finitely many
coefficients of the slab function by
\[
  N_d=-\frac1d\sum_{k=1}^{2d-1}\binom{2d-1}{k}W_{d,k}.
\]
\end{maintheorem}

This combines Corollary~\ref{Cor: wall versus closed comparison} and
Theorem~\ref{Thm: finite wall formula}.

Let $\check X_t$ be the intrinsic mirror family and $\Omega$ its canonical
holomorphic volume form. For $0<\lambda\ll1$, let $\check X_t^{>0}(\lambda)$ be
the compact truncation of its positive real locus defined in \S\ref{Sec:
positive real locus}. In a maximal degeneration of compact Calabi--Yau
threefolds, the positive real locus provides a singular $3$-chain in $W_6$ of
the monodromy weight filtration. The multi-valued period for this $3$-chain thus
provides complete information about the variation of Hodge structures in the
coefficients of $(\log t)^i$, $i=0,\ldots,3$. In our situation of canonical
coordinates with a normalized holomorphic volume form, the terms with $i=2,3$
are topological, while $i=1$ or $i=0$ capture the mirror Gromov--Witten
potential. In our present non-compact situation we use a renormalized version of
this statement that we prove directly.

\begin{maintheorem}
\label{Thm: intro period comparison} Under Assumption~\ref{Ass: analyticity},
the polynomial in Theorem~\ref{Thm: intro wall closed comparison} has degree
$2d-1$ and for every $d>0$ and $\lambda>0$ sufficiently small, we have
\[
  [t^d\log t]\int_{\check X_t^{>0}(\lambda)}\Omega
  =\sum_{k>0}W_{d,k}\lambda^k
   -\sum_{i=1}^3\sum_{k=1}^d W_{d-k,k\be_i}\lambda^{-k}.
\]
Thus the positive powers of $\lambda$ give the generating function of the
$W_{d,k}$. The finite sum involves the opposite outgoing contact orders, in
the notation of \S\ref{SS: wall invariants}.

The coefficient of $t^d\log t$ admits a rational continuation in $\lambda$,
with possible finite poles only at $0$ and $-1$. This continuation satisfies
\[
  \lim_{\lambda\to\infty}
    [t^d\log t]\int_{\check X_t^{>0}(\lambda)}\Omega
  =\sum_{k=1}^{2d-1}\binom{2d-1}{k}W_{d,k}
  =-dN_d.
\]
\end{maintheorem}

This follows from Theorem~\ref{Thm: logarithmic period coefficient} and
Proposition~\ref{Prop: asymptotic period}, together with Theorem~\ref{Thm: intro
wall closed comparison}. The statement on the polynomial degree follows by
direct computation from \eqref{Eqn: W_{d,k}}, which also
gives an explicit formula for $W_{d,k}$. The constant term $d=0$ does not have
enumerative significance, but can be computed to give $3\zeta(2)$ in accordance
with the Gamma conjecture \cite{AGIS}.

For the final statement, let $\cT_{d,k}^{\,i}$ be the sets of rooted trees
defined in \S\ref{Sec: finite wall formulas}. Their leaves are labeled by
$x,y,z$, with $d+k$ leaves of the $i$-th label and $d$ of each other label. An
alternative interpretation is as the set of isomorphism classes of types of
tropical disks in $\RR^2$ with the stated number of unbounded legs of
multiplicity~$1$ in directions $(1,0)$, $(0,1)$, $(-1,-1)$. Write $\widehat
V_\Gamma$ for the set of non-leaf vertices and $\Aut(\Gamma)$ for the set of
automorphisms preserving the root and the leaf labels. The following
answers~(C).

\begin{maintheorem}
\label{Thm: intro tree formula}
Assume that $f$ agrees with the normalized slab function $1+x+y+z+g(t)$ from the
construction of \cite{affinecomplex}, as in Assumption~\ref{Ass: explicit slab
shape}. Then, for every $d>0$,
\[
  N_d=-\frac1d\sum_{i=1}^3\sum_{k=1}^{2d-1}
    \frac1k\binom{2d-1}{k}
    \sum_{\Gamma\in\cT_{d,k}^{\,i}}
      \frac{(-1)^{|\widehat V_\Gamma|}}{|\Aut(\Gamma)|}.
\]
\end{maintheorem}

This theorem is stated as Corollary~\ref{Cor: finite tree formula} below.
Theorem~\ref{Thm: explicit closed formula} shows that this structural result
provides the same formula that Coates and Iritani previously derived via
Givental's formalism \cite[\S10.2.1]{CoatesIritani}.
\medskip

The paper is organized as follows. After constructing a maximal normal crossings
degeneration of $K_{\PP^2}$ in \S\ref{Sec: Mumford degeneration}, an easy
computation gives the truncated period integral in \S3 for small $\lambda$.
Assuming polynomiality of $(-1)^kW_{d,k}$ in $k$, we obtain a rational
continuation in $\lambda$ of the coefficient of $t^d\log t$. Its limit at
infinity is the negative of the constant term of this polynomial. The period
comparison thus reduces to proving polynomiality and identifying the constant
term.

In \S4 we split the contributing tropical types along the edges adjacent to the
root vertex. The logarithmic gluing formula concentrates all dependence
on $k$ in an invariant of maps to a single rational ruled component $Y_i$, a
third Hirzebruch surface blown up at most twice. The remaining factors are
common to the punctured and closed invariants.

Section~5 treats the invariant associated to this surface by relative virtual
localization for a $\GG_m$-action along its ruling. The argument only requires
an action on $Y_i$ compatible with its induced logarithmic structure, although
our model supplies a global action on a Zariski-open subset for convenience. The
component carrying the outgoing puncture is a degree-$k$ cover of a ruling
fiber. Removing this component leaves a moduli problem independent of $k$. After
extracting the sign $(-1)^k$, the localization contribution of the cover is
polynomial in $k$ by nilpotence of the cotangent class at the attaching marking.
The divisor axiom identifies the constant term with the closed central invariant
weighted by its ruling degree. Summing over distinguished vertices then gives
$dN_d$ by additivity of the degree.

Finally, \S\ref{Sec: finite wall formulas} uses finite differences to express
$N_d$ in terms of finitely many wall invariants. Under the additional
identification with the slab function of \cite{affinecomplex}, this gives the
tree sum and the agreement with Givental's formula.

\begin{acknowledgement}
This paper grew out of a broader joint project with Helge Ruddat \cite{vGRS2}.
We thank him warmly for his contributions to its development, including joint
computations at Oberwolfach, providing a proof of a curious algebraic identity,
and discussions of embeddings in compact Calabi--Yau geometries, period
asymptotics,  tree expansions, and logarithmic Gromov--Witten invariants. We
also thank him for agreeing to split the work, with the present paper devoted to
$K_{\PP^2}$ and the broader treatment of the geometries of \cite{ICM} reserved
for a separate paper. We thank Mark Gross for discussions on several aspects of
this work, and Andrea Brini for explanations related to \cite{Brini}.

Research by B.\,Siebert was partially supported by NSF grant DMS-2401174 and a 
Simons Fellowship.
\end{acknowledgement}

\noindent \textbf{Statement on the use of AI.} The derivation of the expression
for $N_d$ in terms of a period integral had been done by January~2025, without
any use of AI systems. Various insights obtained later during that year with
often confusing responses from a number of systems turned out to be irrelevant.
Gemini~3.1 Pro helped to produce a first draft of \S\ref{Sec: wall types and
splitting} and \S\ref{Sec: Localization on Y_i}. ChatGPT-5.6 and~6 have been
used in the final writing phase. All content is entirely the responsibility of
the human authors.


\begin{conventions}
All Gromov--Witten moduli spaces and invariants in this paper are absolute
except briefly in \eqref{Eqn: closed decomposition}.
\end{conventions}


\section{The maximally unipotent degeneration}
\label{Sec: Mumford degeneration}

Intrinsic mirror symmetry for Calabi--Yau varieties takes as input a maximally
degenerating family \cite{associativity,CanScat}. We first construct a
two-parameter toric model adapted to the enumerative computation and then relate
it to the maximally degenerate local $\PP^2$ model of \cite[\S5]{invitation}.

\subsection{The toric construction}

Let $M=\ZZ^2$. Consider the integral polyhedral decomposition of
$M_\RR=\RR^2$ depicted in Figure~\ref{Fig: toric model}. It consists of two
toric diagrams for $K_{\PP^2}$ inducing polarizations of degrees $3$ and $9$ on
$\PP^2$, translated by $(1,-1)$ against each other. The vertices of the two
triangles are at $(1,1)$, $(1,-2)$, $(-2,1)$, $(4,2)$, $(4,-7)$, $(-5,2)$, the
remaining three vertices are at $(-4,2)$, $(2,2)$, $(3,-6)$.

Then let $\varphi$ be the integral piecewise-affine function for the smaller
toric diagram that is identically zero on the inner triangle and changes slope~1
along cell edges, and let $\psi$ be the analogous function for the larger toric
diagram. Together they provide a two-parameter polarized toric degeneration
\[
\PP_\Xi\lra \AA^2=\Spec\CC[\sigma,\tau]
\]
of toric varieties (Mumford degeneration) with momentum polyhedron
\begin{equation}
\label{Eqn: overgraph Xi}
\Xi= \big\{ (m,a,b)\in M_\RR\times\RR\times\RR\,\big|\, a\ge \varphi(m),\,
b\ge \psi(m)\big\}.
\end{equation}
The two degeneration parameters are the monomials $\sigma=\chi^{(0,0,0,1)}$ and
$\tau=\chi^{(0,0,1,0)}$.

\begin{figure}
\input{TwoParamMumford.tikz}
\caption{Polyhedral decomposition for the two-parameter Mumford degeneration
with indicated curve classes.}
\label{Fig: toric model}
\end{figure}
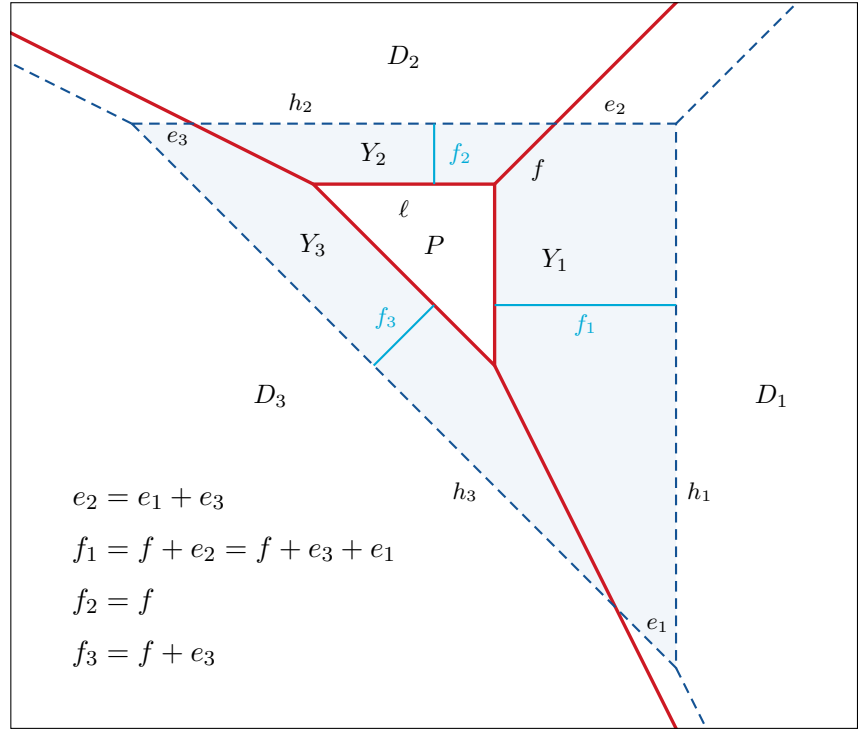

\begin{proposition}
\label{Prop: Mumford degeneration}
The total space $\PP_\Xi$ is smooth and has trivial canonical bundle. The
composition
\[
  \pi\colon\PP_\Xi\lra\AA^2\lra\AA^1=\Spec\CC[\tau]
\]
is a normal crossings degeneration of non-compact Calabi--Yau threefolds with
general fiber $K_{\PP^2}$. Its central fiber has four irreducible components
$P,D_1,D_2,D_3$ corresponding to the four regions bordered by the solid red
lines. Here $P\simeq\PP^2\times\AA^1$, while each $D_i$ is a toric blow-up of
$\Tot\big(\O_{\PP^1}(-3)\big) \times\AA^1$.

Restricting to $\sigma=0$ induces one-parameter toric degenerations with
parameter $\tau$ for the polarized toric varieties defined by the regions
bordered by dashed lines. The component $Y$ corresponding to the bounded dashed
region is a toric degeneration of $\PP^2$ whose central fiber is the union of
$\PP^2$ and three torically blown-up Hirzebruch surfaces $Y_1,Y_2,Y_3$. Moreover,
$\PP_\Xi$ is the total space of $N_{Y/\PP_\Xi}\simeq\O_Y(-D_\infty)$
with zero section $Y$, where $D_\infty\subset Y$ is the degeneration of the
three coordinate lines given by the boundary of the dashed triangle.
\end{proposition}

\begin{proof}
Smoothness is obvious at the trivalent vertices. At the four-valent vertex of
$\Xi$ projecting to $(2,2)$, the integral generators of the edges form the
lattice basis $(1,0,1,0)$, $(1,1,1,1)$, $(-1,0,0,0)$, $(-1,-1,-1,0)$, and
similarly at the two other four-valent vertices.

For $K_{\PP_\Xi}=0$, observe that $\PP_\Xi$ has eight toric divisors, four
contained in $\{\sigma=0\}$ and four contained in $\{\tau=0\}$, and $\sigma\tau$
vanishes to order one on each of them. Equivalently, $(0,0,1,1)\in M\oplus\ZZ^2$
pairs to $1$ with each primitive generator of the normal fan of $\Xi$.

Inverting $\tau$ gives the toric variety with momentum polyhedron $\RR$ times
the overgraph of $\psi$. Since that overgraph is the standard momentum
polyhedron of $K_{\PP^2}$, the general fiber of $\pi$ is toric $K_{\PP^2}$. The
statements about $\pi^{-1}(0)=V(\tau)$ and $V(\sigma)$ follow by restricting the
polyhedral decomposition to the corresponding regions.

Every vertex of $\Xi$ lies on the facet defining $Y$, so every maximal cone of
the normal fan contains its ray. Projection along this ray gives the line-bundle
projection
\begin{equation}
\label{Eqn: Mumford normal bundle}
  p\colon\PP_\Xi\simeq\operatorname{Tot}(N_{Y|\PP_\Xi})\longrightarrow Y.
\end{equation}
The monomial $\sigma$ is fiber-linear with zero divisor $Y+p^*D_\infty$. It
therefore induces a section of $N_{Y|\PP_\Xi}^\vee$ with zero divisor $D_\infty$,
proving $N_{Y|\PP_\Xi}^\vee\simeq\O_Y(D_\infty)$.
\end{proof}

The blown up Hirzebruch surfaces $Y_i$ have general fibers $f_i$, exceptional
curves $e_i$, and strict transforms $h_i$ of the positive sections. For example,
$Y_1$ is obtained from $\FF_3=\PP_{\PP^1}(\O_{\PP^1}\oplus\O_{\PP^1}(3))$ by
blowing up the two points where the positive toric section meets the fibers over
$0$ and $\infty$.

We endow $\PP_\Xi$ and $Y$ with the divisorial log structures defined by
$V(\tau)=P\cup D_1\cup D_2\cup D_3$ and $Y_0=Y\cap V(\tau)$, respectively. Then
$\tau\colon\PP_\Xi\to\AA^1$ and $Y\to \AA^1$ are log smooth for the toric log
structure on $\AA^1$. The divisor $D_\infty\subset Y$ is not part of the log
structure for $Y$. Since $\tau$ is pulled back from $Y$, we have
$V(\tau)=p^{-1}Y_0$. The differentials of $p$ and its zero section give the
canonical equivariant splitting
\begin{equation}
\label{Eqn: zero section tangent splitting}
  \Theta_{\PP_\Xi}|_Y\simeq\Theta_Y\oplus N_{Y|\PP_\Xi},
  \qquad \Theta_Y=T_Y(-\log Y_0).
\end{equation}

We end the discussion of this Mumford degeneration by noting that $\pi:
\PP_\Xi\to \AA^1$ factors via the simpler, but more singular toric variety
$\PP_{\Xi'}$ with momentum polyhedron $\Xi'\subset M_\RR\times\RR^2$ defined as
$\Xi$ by \eqref{Eqn: overgraph Xi}, but with $\psi$ replaced by
$\psi':=\varphi$. This removes the dashed part in Figure~\ref{Fig: toric model}
and makes $(\varphi,\psi')$ change slope by $(1,1)$ along each edge of the red
graph. As a one-parameter Mumford degeneration, one obtains the subdivision of a
three-dimensional polyhedron in $M_\RR\times\RR$ depicted on the left in
Figure~\ref{Fig: toric degen KP2} below. Continue to denote $\sigma,\tau$ as the
monomials corresponding to the tangent vectors $(0,(0,1))$ and $(0,(1,0))$ in
$M\times\ZZ^2$.

\begin{proposition}
\label{Prop: P_Xi->P_ol Xi}
The degeneration $\pi: \PP_\Xi\to \AA^1$ factors over a proper small toric 
morphism
\[
\PP_\Xi\lra \PP_{\Xi'}
\]
that is an isomorphism away from the curve $Z\cap V(\sigma,\tau)$ where $Z$ is
the singular locus of the central fiber $V(\tau)\subset \PP_{\Xi'}$, the union 
of its six toric curves, three $\PP^1$'s and three $\AA^1$'s.
\end{proposition}

\begin{proof}
The functions $\varphi,\psi$ have the same four slopes, with zero slope on the
bounded regions. At every vertex of the superposition in Figure~\ref{Fig: toric
model}, the combined slopes consist of zero and two nonzero slopes that also
meet at a vertex of the red triangle. Since all cells are pointed, this verifies
\eqref{Eqn: mixed incidence} for $\psi_0=\varphi$ and $\psi_1=\psi$.
Lemma~\ref{Lem: toric small modification} therefore gives a proper small
modification $\PP_\Xi\to\PP_{\Xi'}$ over $\AA^2$. The exceptional locus consists
of the surfaces defined by the graph faces over the regions where the slopes of
$\psi_0=\varphi$ and $\psi_1=\psi$ differ, namely the three bounded shaded
regions and the three unbounded strips. Their images are the toric curves
corresponding to the three bounded and three unbounded red edges, forming $Z\cap
V(\sigma,\tau)$. On the bounded surfaces the morphism is the ruling
$Y_i\to\PP^1$, and on the strip surfaces it is the projection
$\PP^1\times\AA^1\to\AA^1$.
\end{proof}

\begin{remark}
\label{Rem: blowing up D_i's}
The morphism $\PP_\Xi\to\PP_{\Xi'}$ can also be obtained by blowing up $D_1$,
then the strict transform of $D_3$, and finally the strict transform of $D_2$.
These blow-ups correspond to successive inward parallel translations of the
supporting hyperplanes of the three toric divisors.
\end{remark}

\begin{lemma}
\label{Lem: toric small modification} Let $M$ be a lattice and let
$\varphi,\psi_0,\psi_1$ be convex integral piecewise affine functions on $M_\RR$
with finitely many domains of linearity. Suppose that $\psi_0,\psi_1$ have the
same sets of slopes. Denote the maximal domains of $\varphi$ by $\sigma_i$ and
those of $\psi_k$ by $\tau_j^k$, labeled so that $\psi_0|_{\tau_j^0}$ and
$\psi_1|_{\tau_j^1}$ have the same slope. Put
\[
  \Xi_k=\big\{(m,a,b)\in M_\RR\times\RR^2\mid
    a\geq\varphi(m),\ b\geq\psi_k(m)\big\},\qquad k=0,1.
\]
The identity on the common torus extends to a morphism
$\PP_{\Xi_1}\to\PP_{\Xi_0}$ if and only if
\begin{equation}
\label{Eqn: mixed incidence}
  \sigma_I\cap\tau_J^1\neq\emptyset
  \quad\Longrightarrow\quad
  \sigma_I\cap\tau_J^0\neq\emptyset
\end{equation}
for all index sets $I,J$, where subscripts denote intersections and the empty
intersection is $M_\RR$. Whenever it exists, this morphism is a proper small
toric modification over $\AA^2$.
\end{lemma}

\begin{proof}
Let $p_i,q_j$ denote the slopes on $\sigma_i,\tau_j^k$. The facets of $\Xi_k$
have primitive inward normals $(-p_i,1,0)$ and $(-q_j,0,1)$ and project onto
$\sigma_i$ and $\tau_j^k$, respectively. Thus the facets indexed by $I,J$ meet
if and only if $\sigma_I\cap\tau_J^k\neq\emptyset$. A normal cone is generated
by the normals of its incident facets. Since these normals generate rays of both
fans, a source cone lies in a target cone if and only if the corresponding
target facets meet. Hence cone containment is equivalent to \eqref{Eqn: mixed
incidence}. The two fans have the same support, the cone generated by their
common rays, so the resulting morphism is proper. It is small because no rays
are added and commutes with the two coordinate monomials because it is induced
by the identity on the lattice.
\end{proof}


\subsection{Embedding in a maximal degeneration}
\label{SS: maximal degeneration}

The degeneration $\PP_\Xi\to\AA^1_\tau$ constructed above is adapted to the curve
counting problem, but it is not maximal, since its central fiber has no
zero-dimensional strata. For the intrinsic mirror construction, we instead
identify a Zariski-open neighborhood of $V(\sigma)$, which contains the zero
section of $K_{\PP^2}$ for $\tau\neq0$, with a small resolution of the toric 
degeneration
\[
\ol\pi: \ol X\lra \AA^1_\tau
\]
of $K_{\PP^2}$ in \cite[Expl.\,5.1]{invitation}. Recall from loc.cit.\ that the
central fiber of $\ol\pi$ has six irreducible components, a unique compact one
$\PP^1\times\PP^2$, two copies $P_+$, $P_-$ of $K_{\PP^2}$ and three copies of
$\Tot(\O_{\PP^1\times\PP^1}(-1,-3))$, arranged as indicated in the middle of
Figure~\ref{Fig: toric degen KP2}. The general fiber is isomorphic to
$K_{\PP^2}$ minus the graph of the rational $2$-form with logarithmic poles
along the coordinate lines. In affine coordinates $r,v,z,s,\tau$ centered at a
$0$-stratum, say at $P_+\cap P\cap D_1\cap D_2$ with $r,v,z,s$ respectively
defining $P, D_1,D_2,P_+$, the family is given by
\[
rvzs=(1+s)\tau.
\]
Rewriting this equation as $(rvz-\tau)s=\tau$ shows that $\tau\neq0$ implies
$rvz\neq \tau$; after eliminating $s$, a solution $(r,z,v)$ can be identified
with $r dz\wedge dv$ in the fiber of $K_{\PP^2}$ over $[z,v,1]\in\PP^2$. The
removed section is $r=\frac{\tau}{zv}$. Note also that $s$ is defined on
$\ol X\setminus (P_+\cup P_-)$.

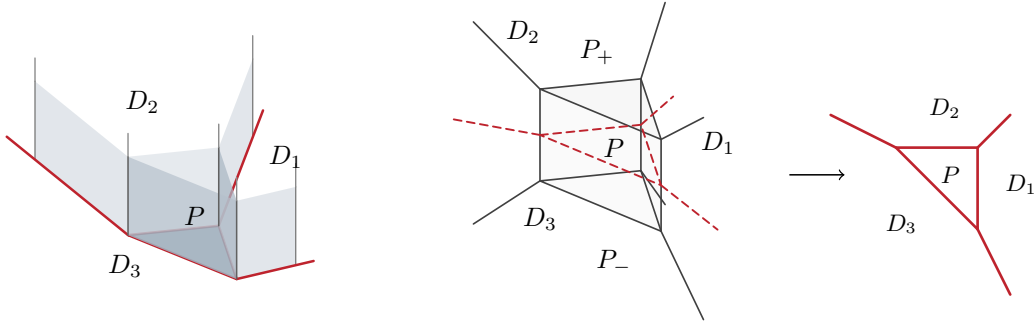
\begin{figure}
\input{MumfordInvitationWeb.tikz}
\caption{Left: The intersection complex for the singular Mumford degeneration 
with $\psi=\varphi$. Middle and right: The intersection complex for the family 
from \cite{invitation} and the affine projection to the web diagram of 
$K_{\PP^2}$.}
\label{Fig: toric degen KP2}
\end{figure}

\begin{proposition}
\label{Prop: Mumford vs invitation model}
The morphism of Proposition~\ref{Prop: P_Xi->P_ol Xi} restricts to a
factorization over $\AA^1_\tau$
\[
\PP_\Xi\setminus V(\sigma-1)\longrightarrow
\PP_{\Xi'}\setminus V(\sigma-1)\stackrel{\sim}{\longrightarrow}
\ol X\setminus(P_+\cup P_-).
\]
Under the second arrow, $1+s$ pulls back to $\sigma$. The composite is
proper and is an isomorphism away from
$\ol X_{\sing}\cap V(1+s)$.
\end{proposition}

\begin{proof}
Identifying $M=\ZZ^2$ according to Figure~\ref{Fig: toric model}, put
$r_0=\chi^{(1,1,1,1)}$. In the toric affine chart at the zero-stratum
$D_1\cap D_2\cap P\cap V(\sigma)$ of $\PP_{\Xi'}$, we have $r_0vz=\sigma\tau$.
Since $\sigma-1$ is invertible, the substitutions
\[
s=\sigma-1,\qquad r=\frac{r_0}{\sigma-1}
\]
give $rvzs=(1+s)\tau$. Their inverse is $\sigma=1+s$, $r_0=rs$. Cyclic analogues
provide the extensions to the other two affine charts covering $\PP_{\Xi'}$.

The first arrow is obtained from the proper morphism of Proposition~\ref{Prop:
P_Xi->P_ol Xi} by base change. By Proposition~\ref{Prop: P_Xi->P_ol Xi}, the
image of the exceptional locus of the first arrow is locally
$V(\sigma,\tau,r_0v,r_0z,vz)$. Under the stated substitutions this subset
becomes $V(1+s,\tau,rv,rz,vz)$, the singular locus of $\ol X$ on this chart. The
same holds cyclically, proving the last assertion.
\end{proof}

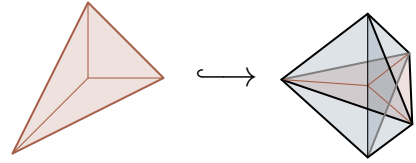
\begin{wrapfigure}[12]{r}{.45\textwidth}
  \vspace*{.5ex}
  \centering \input{P2local.tikz}
  \begin{minipage}{.4\textwidth}
  \caption{The height-one slice of the tropicalization of $X$, a polyhedral
  decomposition of a bipyramid over a triangle into six simplices, and its
  central planar subcomplex.}
  \label{Fig: height-one slice}
  \end{minipage}
\end{wrapfigure}

We finally define the maximally degenerating family with normal crossings 
singularities
\[
\pi: X\lra \AA^1
\]
by gluing $\ol X\setminus \big(\ol X_\sing\cap V(1+s)\big)$ to $\PP_\Xi\setminus
V(\sigma-1)$ via the isomorphism in Proposition~\ref{Prop: Mumford vs invitation
model}. Note that by Remark~\ref{Rem: blowing up D_i's}, $X$ can also be
obtained by blowing up $\ol X$ first in $D_1$, then in the strict transform of
$D_3$, then in the strict transform of $D_2$.

Let
\[
  \rho\colon\Sigma(X)\longrightarrow\Sigma(\AA^1)\hspace{40ex}
\]
be the tropicalization of $\pi$,  put $B=\Sigma(\pi)^{-1}(1)$ and denote by
$B_0\subset B$ the central planar subcomplex as illustrated in Figure~\ref{Fig:
height-one slice}.

\noindent Since every irreducible one-dimensional stratum contained in $P$
deforms away from the log-singular locus into the interior of the toric
component $P$, all change-of-chart morphisms are toric, and $B$ as an affine
manifold is the fan of $\PP^1\times\PP^2$ cut off at integral distance one and
subdivided into six simplices, see Figure~\ref{Fig: height-one slice}. The
splitting argument of Lemma~\ref{Lem: vanishing from splitting} will show that
the tropicalizations of all contributing wall types and rigid closed types are
contained in $B_0$. In particular, $B_0$ contains all nontrivial walls of the
canonical wall structure from \cite{CanScat}.


\subsection{Compact curve classes}

For later use we compute intersection numbers for the curve classes displayed in
Figure~\ref{Fig: toric model}. The class $\ell$ is the class of a line in the
central component $P$. The classes $f_1,f_2,f_3$ are the fiber classes in the
blown up Hirzebruch surfaces $Y_1,Y_2,Y_3$, and $e_1,e_2,e_3$ are the indicated
exceptional classes in reducible fibers. Defining $f= f_2$ and taking into
account the obvious relations stated in Figure~\ref{Fig: toric model}, the monoid
of effective curve classes relevant for the intrinsic mirror computations is
\begin{equation}
\label{Eqn: effective cone}
  \NE\big(X_0\setminus (P_-\cup P_+)\big)= \NE\big(V(\sigma,\tau)\big)=
  \NN\ell+\NN f+\NN e_1+\NN e_3.
\end{equation}
The component $P$ contains many further compact curves not contained in
$V(1+s,\tau)= V(\sigma,\tau)$. The assertion needed below is only that curves
not contained in $V(\sigma,\tau)$ do not occur in contributing wall types or
rigid closed types.  This follows from the splitting argument of Lemma~\ref{Lem:
vanishing from splitting}, rather than from compactness alone.  Consequently,
for the wall structure and the computation of $N_d$, it suffices to consider the
curve classes represented in $V(\sigma,\tau)$.

The following table gathers the intersection numbers between the relevant 
divisors of the central fiber $X_0$ of $X\to \AA^1$ with these curve classes.
\begin{center}
\begin{tabular}{c|rrrr}
\toprule
 & $P$ & $D_1$ & $D_2$ & $D_3$ \\
\midrule
$\ell$ & $-3$ & $1$ & $1$ & $1$ \\
$e_1$  & $0$  & $-1$ & $0$ & $1$ \\
$e_3$  & $0$  & $0$ & $1$ & $-1$ \\
$f$    & $1$  & $0$ & $-1$ & $0$ \\
\bottomrule
\end{tabular}
\end{center}
The sum of every row vanishes, in agreement with the fact that $P+D_1+D_2+D_3$
is a principal divisor in a neighborhood of $s=-1$. Note $P\cdot \ell=-3$
implies $\O_X(P)|_{P\cap V(s+1)}\simeq\O_{\PP^2}(-3)$, as expected.

The positive toric sections of the three $Y_i$ have classes
\begin{equation}
\label{Eqn: positive sections}
  h_1=\ell+3f+e_1+2e_3,
  \qquad
  h_2=\ell+3f,
  \qquad
  h_3=\ell+3f+2e_3.
\end{equation}
The unions
\[
  \widetilde h_1=h_1,
  \qquad
  \widetilde h_2=h_2+e_2+e_3,
  \qquad
  \widetilde h_3=h_3+e_1
\]
represented by the three edges of the dashed triangle in Figure~\ref{Fig: toric
model} are homologous. Their common class is
\begin{equation}
\label{Eqn: specialization of a line}
  \beta=\ell+3f+e_1+2e_3=\ell+f_1+f_2+f_3.
\end{equation}
It is the specialization of the class of a line in the zero section of
$K_{\PP^2}$ in $X_\tau\subset X$ for $\tau\neq 0$.


\section{Renormalized periods over the positive real locus}

\subsection{Intrinsic mirror geometry}
\label{Sec: intrinsic mirror geometry}

As discussed at the end of \S\ref{SS: maximal degeneration}, all nontrivial
canonical walls are contained in $B_0$. We work in the chart at the central
vertex $v_0$, and write $t$ for the curve-degree parameter defined by the
pullback of the hyperplane class on $\PP^2$. The product of the canonical wall
functions\footnote{Peter Zaika \cite{Zaika} has shown that the construction of
\cite{CanScat} remains valid in the case where $X$ is only proper over an affine
variety, as in our situation.} from \cite{CanScat}, interpreted in this chart,
is
\begin{equation}
\label{Eqn: Canonical slab function}
f=\exp\Big(\sum_{d\ge 0,\ m\in\ZZ^2\setminus\{0\}}
  |m|\,W_{d,m}\,t^d\chi^{-m}\Big).
\end{equation}
Here $W_{d,m}$ is the aggregate invariant with outgoing contact order $m$ and
curve degree $d$, and $\chi^{-m}$ is the corresponding tangent theta monomial,
$|m|$ its integral length. We refer to \S\ref{SS: wall invariants} for the
monomial convention, the multiplicities in the aggregation, and the passage from
the individual walls to $f$. We interpret the wall function and its logarithm in
the completed ring $\QQ\llbracket t,x,y,z\rrbracket/(xyz-t)$. For the explicit
formulas in \S\ref{Sec: finite wall formulas}, we work with the expected
expression
\begin{equation}
\label{Eqn: expected slab fct}
f=1+x+y+z+g(t)
\end{equation}
with $g$ an analytic function in $t$ determined by the normalization condition,
as in \cite[Expl.\,5.1]{invitation} and \cite[\S5]{ICM}. Analyticity of $g$ has
been shown in \cite[Cor.\,6.21]{CCLT} and \cite[Thm.\,1.1, Expl.\,7.2]{Lau}. It
also follows directly from a simple analytic equation $g$ fulfills and the
analytic inverse function theorem. For now we only make the following assumption.

\begin{assumption}
\label{Ass: analyticity} The canonical wall function $f$ converges near
$x=y=z=t=0$ and admits a single-valued holomorphic continuation to a
neighbourhood of
\[
  \big\{(x,y,z,t)\in\RR_{>0}^4\,\big|\,xyz=t<\epsilon\big\}\subset \CC^4
\]
for some $\epsilon>0$.
\end{assumption}

Since there is no scattering in $B_0$, $x,y,z$ agree with the
generalized theta functions $\vartheta_{(1,0,0)}$, $\vartheta_{(0,1,0)}$,
$\vartheta_{(-1,-1,0)}$ \cite{GHS}, and they fulfill the toric relation $xyz=t$.
Denoting further by $u,v$ the generalized theta functions
$\vartheta_{(0,0,-1)},\vartheta_{(0,0,1)}$, the expansion of $u$ in terms of
broken lines at a point slightly below $B_0$ gives the mirror
geometry as an analytic family for small $t$:
\begin{equation}
\label{Eqn: intrinsic mirror central chart}
  \check X_t\colon\qquad
  uv=f\cdot t,\qquad xyz=t.
\end{equation}
For fixed $t\neq0$, the second equation gives $z=t/(xy)$, so that we obtain a
family of conics in the $(u,v)$-plane with parameters $x,y\in\CC^*$. The
canonical holomorphic volume form coming with the construction, for the
orientation of $B_0$ shown in Figure~\ref{Fig: height-one slice} and restricted
to the locus $uxy\neq0$ is
\begin{equation}
\label{Eqn: intrinsic mirror volume form}
  \Omega=\frac{du}{u}\wedge\frac{dx}{x}\wedge\frac{dy}{y}
  =\dlog u\wedge \dlog x \wedge \dlog y.
\end{equation}

The calculation below proceeds term by term in $\log f$. We will show
that only terms $W_{d,m}$ with contact orders $m=-k\be_i$, $k\in\ZZ$, and
\begin{equation}
\label{Eqn: e_i}
  \be_1=(1,0),\quad \be_2=(0,1),\quad \be_3=(-1,-1),
\end{equation}
contribute to coefficients of the terms in the period integral with enumerative
meaning. Note the $\be_i$ are the primitive generators of the three edges in
$B_0$ emanating from $v_0$.


\subsection{Period integrals over compact positive real loci}
\label{Sec: positive real locus}

Let $0<\lambda\ll 1$ be such that $f$ converges for $0\le x,y,z,t^{1/3}< \lambda$.
For $0<t\ll 1$, we will compute the integral of the canonical holomorphic volume
form $\Omega$ over the compact subset
\[
  \check X_t^{>0}(\lambda)
  =\big\{(u,v,x,y,z)\in \check X_t^{>0}\,\big|\,
       u,v\leq1,\ x,y,z\leq \lambda\big\}
\]
of the positive real locus $\check X_t^{>0}\subset \check X_t$ as a function of
$\lambda$. The motivation for the cutoff in $u,v$ is that embedding our
degeneration of $\PP^2\subset K_{\PP^2}$ into a degeneration of Calabi--Yau
threefolds, the dependence on that cutoff reflects curve classes not contained
in $\PP^2$. The cutoff in $\lambda$ is necessary because the expansion of $\log
f$ in \eqref{Eqn: Canonical slab function} converges only for small $x,y,z,t$,
and we will eventually extract the enumerative information from a holomorphic
continuation for $\lambda\to\infty$. We refer to \cite{vGRS2} for a discussion of 
this cutoff.

Since $f$ has real coefficients and constant term one, it is positive for
$x,y,z$ and $t$ sufficiently small. Taking $\log$ of \eqref{Eqn: intrinsic
mirror central chart} therefore allows us to eliminate $v$ and to identify the
fibre of the projection to the $(x,y,z)$-coordinates with the interval $\log
u\in \big[\log t+\log f,0\big]$. Working in logarithmic coordinates and
integrating over $\log u$ gives
\begin{equation}
\label{Eqn: First integral}
\begin{aligned}
  \int_{\check X_t^{>0}(\lambda)}\Omega
  &= \int_{\Gamma(\lambda)} \int_{\log t+\log f}^0
    \dlog u\wedge \dlog x \wedge \dlog y\\
  &= -\int_{\Gamma(\lambda)}
    \Big(\log t+\log f\big(x,y,{\textstyle \frac{t}{xy}},t\big)\Big)
    \dlog x\wedge\dlog y.
\end{aligned}
\end{equation}
Here
\[
\Gamma(\lambda)=\big\{(\xi,\eta)\in\RR^2\,\big|\, \log t-2\log \lambda\le 
\xi,\eta,\log t-\xi-\eta\le \log\lambda\big\}
\]
is the triangle-shaped domain of integration after elimination of $z=t/(xy)$ in
logarithmic coordinates.

We are interested in the coefficient of $t^d\log t$ for $d>0$, as a function of
$\lambda$. We interpret the wall invariants $W_{q,m}$ with outgoing contact
order $m$ in an endpoint cell whose tangent cone at $v_0$ contains $-m$ and 
include degree-zero invariants.

\begin{theorem}
\label{Thm: logarithmic period coefficient}
For $d>0$ and $\lambda>0$ sufficiently small, with $\be_i$ as in
\eqref{Eqn: e_i}, we have
\begin{equation}
\label{Eqn: logarithmic period coefficient}
  [t^d\log t]\int_{\check X_t^{>0}(\lambda)}\Omega
  =\sum_{i=1}^3\left(
    \sum_{k>0}\lambda^k W_{d,-k\be_i}
    -\sum_{k=1}^d\lambda^{-k}W_{d-k,k\be_i}
  \right).
\end{equation}
\end{theorem}

\begin{proof}
Convergence of $\log f$ near the origin permits termwise integration. Monomials
$x^a y^b$ in the expansion of $\log f$ in \eqref{Eqn: Canonical slab function}
with $ab(a-b)\neq0$ do not contribute to $t^d\log t$:
\begin{multline*}
  \int_{\Gamma(\lambda)}t^r x^ay^b\,\dlog x\wedge\dlog y
  =t^r\int_{t/\lambda^2}^{\lambda}\int_{t/(\lambda x)}^{\lambda}
    x^{a-1}y^{b-1}\,dy\,dx\\
  =\frac{t^r}{b}\int_{t/\lambda^2}^{\lambda}
    \left(\lambda^bx^{a-1}
      -t^b\lambda^{-b}x^{a-b-1}\right)\,dx
  =\frac{t^r\lambda^{a+b}}{ab}
    +\frac{t^{r+a}\lambda^{b-2a}}{a(a-b)}
    -\frac{t^{r+b}\lambda^{a-2b}}{b(a-b)}.
\end{multline*}
The remaining monomials can be written $t^r x^k$, $t^r y^k$,
$t^r z^k=t^{r+k}(xy)^{-k}$ with $k\neq0$. By cyclic invariance of the domain and
the logarithmic measure it suffices to compute
\begin{align*}
  \int_{\Gamma(\lambda)}t^r x^k\,\dlog x\wedge\dlog y
  &=t^r\left[\frac{x^k}{k}
    \left(\log x+2\log\lambda-\log t-\frac1k\right)
    \right]_{t/\lambda^2}^{\lambda}\\
  &=-\frac{t^r\lambda^k}{k}\log t
    +t^r\lambda^k\left(\frac{3\log\lambda}{k}-\frac1{k^2}\right)
    +\frac{t^{r+k}\lambda^{-2k}}{k^2}.
\end{align*}
At the lower endpoint, $\log x+2\log\lambda-\log t=0$, leaving the last term
in the displayed expression.

For $k>0$, tropicalizations of wall type with contact orders $-k\be_1$ and
$k\be_1$, taken with endpoints in cells whose tangent cones at the central
vertex contain $\be_1$ and $-\be_1$, respectively, contribute the terms
\[
  \lambda^k W_{d,-k\be_1}\,t^d\log t,\qquad
  -\lambda^{-k}W_{d,k\be_1}\,t^{d+k}\log t.
\]
The case $a=b=0$ is excluded from $\log f$ by the definition of canonical wall
functions, while the remaining term in \eqref{Eqn: First integral} is
\[
  -\int_{\Gamma(\lambda)}\log t\,\dlog x\wedge\dlog y
  =-\frac12(\log t)^3+3\log\lambda\,(\log t)^2
    -\frac92(\log\lambda)^2\log t.
\]
Extracting the coefficient of $t^d\log t$ for $d>0$, with the integral length
$|k\be_i|=k$ and the minus sign in \eqref{Eqn: First integral}, gives
\eqref{Eqn: logarithmic period coefficient}.
\end{proof}

\begin{remark}
Redoing this computation with the non-normalized naive mirror function $1+x+y+z$
leads to a term $\log\lambda\cdot t^d\log t$ from $a=b=0$ for all $d>0$. With
this term, the asymptotic limit in the following \S\ref{Sec: asymptotic limit}
can still be done, but requires a much more complicated reasoning involving a
Mellin-Barnes integral and subsequent reparametrization by the mirror map to
obtain the Gromov--Witten potential.
\end{remark}


\subsection{Asymptotic limit}
\label{Sec: asymptotic limit}

For $0<\lambda\ll1$, define the degree-$d$ wall generating function
\[
  W_d(\lambda)
  :=\sum_{i=1}^3\sum_{k>0}W_{d,-k\be_i}\lambda^k.
\]
Theorem~\ref{Thm: logarithmic period coefficient} then reads
\[
  [t^d\log t]\int_{\check X_t^{>0}(\lambda)}\Omega
  =W_d(\lambda) -\sum_{i=1}^3\sum_{k=1}^d\lambda^{-k}W_{d-k,k\be_i}.
\]
The following assumption, proved in Corollary~\ref{Cor: wall versus closed
comparison} below, isolates the input needed for the asymptotic limit.

\begin{assumption}[Polynomiality]
\label{Ass: wall polynomiality}
For every $d>0$ and $i=1,2,3$, there is a polynomial $P_{d,i}(s)\in\QQ[s]$ such
that
\[
  P_{d,i}(k)=(-1)^kW_{d,-k\be_i},\qquad k\in\ZZ_{>0}.
\]
\end{assumption}

\begin{proposition}
\label{Prop: asymptotic period}
Under Assumption~\ref{Ass: wall polynomiality}, for every $d>0$ the coefficient
in \eqref{Eqn: logarithmic period coefficient} extends to a rational function of
$\lambda$, with possible finite poles only at $0$ and $-1$, and
\[
  \lim_{\lambda\to\infty}
  [t^d\log t]\int_{\check X_t^{>0}(\lambda)}\Omega
  =\lim_{\lambda\to\infty}W_d(\lambda)
  =-\sum_{i=1}^3P_{d,i}(0).
\]
\end{proposition}

\begin{proof}
Write $\Delta P(k)=P(k+1)-P(k)$. The Newton expansion and the binomial
generating series give
\[
  P_{d,i}(k)=\sum_{j\geq0}\binom{k}{j}\Delta^jP_{d,i}(0),\qquad
  \sum_{k\geq0}\binom{k}{j}(-\lambda)^k
  =\frac{(-\lambda)^j}{(1+\lambda)^{j+1}}.
\]
The first sum is finite since $P_{d,i}$ is a polynomial. Hence
\[
  \sum_{k>0}\lambda^kW_{d,-k\be_i}
  =\sum_{j\geq0}\Delta^jP_{d,i}(0)
    \frac{(-\lambda)^j}{(1+\lambda)^{j+1}}-P_{d,i}(0).
\]
The right-hand side gives the asserted rational continuation. Every term in the
last sum tends to zero as $\lambda\to\infty$, as does the finite Laurent
polynomial $\sum_{k=1}^d\lambda^{-k}W_{d-k,k\be_i}$ in \eqref{Eqn: logarithmic
period coefficient}. The stated limit follows.
\end{proof}

Proposition~\ref{Prop: asymptotic period} gives the remarkable result that the
limit $\lim_{\lambda\to\infty} W_d(\lambda)$ captures the constant term of the
polynomial with values $(-1)^{k+1} W_{d,k}$ for $k>0$, thus gives the proper
definition of the meaningless $W_{d,0}$. It is also important to distinguish
this algebraic extension from the logarithmic Gromov--Witten
invariants with one point of contact order $0$ that can be shown to appear as 
coefficients of $g(t)$, cf.\ \cite{Lau,vGRS2}. Verification of
Assumption~\ref{Ass: wall polynomiality} and the identification of these
constant terms with closed invariants are deferred to \S\ref{SS: polynomiality
and assembly}.


\section{Wall types, closed types, and logarithmic splitting}
\label{Sec: wall types and splitting}

\subsection{Wall invariants and closed invariants}
\label{SS: wall invariants}

We use the degeneration and tropicalization of \S\ref{SS: maximal degeneration}
and follow \cite[\S3.2]{CanScat} for the definition of wall invariants. We
assume familiarity with \cite[\S2]{CanScat} which gives a short recap of
punctured Gromov--Witten invariants and the related tropical geometry.

Let $H$ be the pullback of the hyperplane class on $\PP^2$, and put $d(A)=H\cdot
A$ for a compact curve class $A$. Note that for a curve $C\subset
V(\sigma,\tau)$ in the Mumford model $\PP_\Xi$ from Figure~\ref{Fig: toric model},
$d([C])$ can be computed by the intersection of $C$ with (the lift to cohomology
of) the class $\beta$ in \eqref{Eqn: specialization of a line}, the class of a
degeneration of a family of lines in the zero section of $K_{\PP^2}$. This also
immediately shows $d(e_i)=d(f_i)=0$ for the exceptional and fiber curve classes
and $d(h_i)=1$.

As in \cite{CanScat} a decorated tropical type is written $\btau=(\tau,\bA)$,
with $\tau=(G,\bsigma,\mathbf u)$. Here $G$ is the dual graph, $\bsigma$ assigns
cones of $\Sigma$ to its vertices, edges, and legs, $\mathbf u$ records the
contact orders, and $\bA(v)$ is an effective curve class on the stratum
$X_{\bsigma(v)}$. All vertex genera are zero and are suppressed in the notation.
We also write $\tau$ for the cone parametrizing tropical maps of this type and
\[
  A=\sum_{v\in V(G)}\bA(v)
\]
for the total curve class, with the summands pushed forward to $X$. Let
$\scrM(X,\btau)$ be the moduli space of basic stable punctured maps to $X$
marked by $\btau$, as in \cite{ACGS2}. Because $X$ is not compact, we need a more 
detailed analysis to check compactness of the relevant $\scrM(X,\btau)$.

\begin{lemma}
\label{Lem: support of maps with at most one puncture}
The image of a genus-zero stable punctured map to $X_0$ with at most one marked
point is contained in $V(s+1,\tau)$. If its total curve class is not a positive
multiple of one of the exceptional classes $e_1,e_2,e_3$, its image is contained
in the projective surface
\[
  \big(P\cap V(s+1)\big)\cup Y_1\cup Y_2\cup Y_3.
\]
For every fixed decorated type $\btau$ satisfying this additional condition,
$\scrM(X,\btau)$ is proper.
\end{lemma}

\begin{proof}
Let $f:C^\circ\to X_0$ be such a map, let $\nu:X\to\ol X$ be the resolution, and
put $\ol f=\nu\circ f$ and $\ol P=\nu(P)$. Stability forces $f$ to be
nonconstant. Every compact curve in $\ol X_0$ lies in $\ol P$ since the bundles
describing $P_\pm$ and the $\ol D_i$ in \S\ref{SS: maximal degeneration} are
negative, so their proper curves lie in their zero sections, all contained in
$\ol P$. By connectedness, $\ol f(C)$ is therefore contained in $\ol P$ or is a
point outside it.

In the latter case, $f(C)$ lies in a nontrivial fiber of $\nu$. Outside $\ol P$
these are the $\PP^1$-fibers of the unbounded strips in Figure~\ref{Fig: toric
model}, of classes $e_i$, and lie in $V(s+1,\tau)$. In particular, their total
class is a positive multiple of an $e_i$.

It remains to treat the case $\ol f(C)\subset\ol P\simeq\PP^1\times\PP^2$. Then
\[
  a:=\deg(\operatorname{pr}_1\circ\ol f)^*\O_{\PP^1}(1)
  =f_*[C]\cdot P_+=f_*[C]\cdot P_-\geq0.
\]
If $a>0$, the degree--contact relation \cite[Cor.\,2.30]{ACGS2} requires a marked
point mapping to each of $P_+$ and $P_-$. Since these divisors are disjoint,
this contradicts the assumption on the markings. Hence $a=0$, and connectedness
makes $\operatorname{pr}_1\circ\ol f$ constant. Denote its value by
$s_0\in\PP^1$, including the two end values.

If $s_0\neq-1$, the resolution is an isomorphism over the corresponding
$\PP^2$-fiber, so $f$ has image in this fiber and class $d\ell$, with $d>0$. Its
intersection with each $D_i$ is $d$. The degree--contact relation would
therefore require a marked point whose image lies in $P\cap D_1\cap D_2\cap
D_3=\emptyset$. Thus $s_0=-1$ and $f$ has image in $\nu^{-1}\big(\ol 
P\cap V(s+1)\big)$, the stated reducible surface.

For a fixed decorated type satisfying this condition, properness follows from
this projective support by the boundedness and stable reduction arguments of
\cite[Cor.\,3.19]{ACGS2}.
\end{proof}

For a wall type, $G$ is a tree with a single leg $L_{\mathrm{out}}$ of nonzero
contact order $u_\tau=\mathbf u(L_{\mathrm{out}})$, and $\tau$ is realizable and
balanced in the sense of \cite[Def.\,3.6]{CanScat}. If $h$ is the universal
tropical map and $\tau_{\mathrm{out}}$ is the cone in its domain corresponding
to $L_{\mathrm{out}}$, the dimension conditions are
\[
  \dim\tau=2,\qquad \dim h(\tau_{\mathrm{out}})=3.
\]
The wall support is $\mathfrak p_\tau=h(\tau_{\mathrm{out}})\cap B$. On the
slice $B$, a wall type therefore gives a one-dimensional family of tropical
curves whose outgoing legs sweep out a two-dimensional subset.

For wall types, properness follows from Lemma~\ref{Lem: support of maps with at
most one puncture}. Indeed, pure exceptional classes have outgoing-leg images in
boundary edges of $B_0$, contrary to the wall-type dimension condition. The
virtual class has dimension zero by \cite[Lem.\,3.9]{CanScat}. Following
\cite[Constr.\,3.13]{CanScat}, define
\begin{equation}
\label{Eqn: decorated wall invariant}
  W_{\btau}
  =\frac{1}{|\Aut(\btau)|}
    \deg[\scrM(X,\btau)]^{\virt}.
\end{equation}
Here $\Aut(\btau)$ is the automorphism group of the decorated graph.
Automorphisms of the $\btau$-marked stable punctured maps are already accounted
for in the degree of the virtual class.

The multiplicity entering the wall function is the lattice index
\begin{equation}
\label{Eqn: wall lattice index}
  k_\tau
  =\big[\Lambda_{h(\tau_{\mathrm{out}})}:
    h_*(\Lambda_{\tau_{\mathrm{out}}})\big],
\end{equation}
where $\Lambda_{h(\tau_{\mathrm{out}})}$ denotes the integral tangent space of
the real image cone $h(\tau_{\mathrm{out}})$. Writing $q^A$ for the monomial of
curve class $A$, the wall attached to $\btau$ has function
\begin{equation}
\label{Eqn: individual canonical wall}
  f_{\btau}=\exp\big(k_\tau W_{\btau}q^A\chi^{-u_\tau}\big).
\end{equation}
The monomial is interpreted in the local monoid algebra at an interior point of
the wall, with transport between adjacent maximal cells as in
\cite[\S3.1]{CanScat}. This transport can change the $t$-exponent in local
coordinates. For example, $xyz=t$ gives $x=t y^{-1}z^{-1}$ in the chart where
$y$ and $z$ are invertible. Note also that the tangent exponent is the negative
of the outgoing contact order. The construction is first made modulo a cofinite
monomial ideal in the curve-class ring and then passed to the inverse limit.

At a general point $p$ of $B_0$, the attached function is the product of
$f_{\btau}$ over the walls containing $p$. All nontrivial walls are contained in
$B_0$ by Lemma~\ref{Lem: support of maps with at most one puncture}. Their
crossing automorphisms fix its tangent monomials, so consistency identifies
these products on adjacent regions after transport.

To write its expansion independently of an endpoint cell, we use the tangent
vectors $\be_i$ of \eqref{Eqn: e_i}. The decomposition into three sectors gives
for every $m\in\ZZ^2$ a unique expression
\[
  m=a_1\be_1+a_2\be_2+a_3\be_3,\qquad
  a_i\in\NN,\qquad \min\{a_1,a_2,a_3\}=0.
\]
Put $\vartheta_m=x^{a_1}y^{a_2}z^{a_3}$. These are the tangent theta monomials
from \S\ref{Sec: intrinsic mirror geometry}, with $xyz=t$. Choose one of the
three maximal cells of $B_0$ whose tangent cone at $v_0$ contains $-m$, and a
general point $p_m$ in this cell. In that cell, $\vartheta_{-m}$ is the toric
monomial of exponent $-m$. For $d\geq0$ and $m\neq0$, define
\begin{equation}
\label{Eqn: aggregate wall invariant}
  W_{d,m}
  =\frac{1}{|m|}
    \sum_{\substack{\btau,\ p_m\in\mathfrak p_\tau\\
      u_\tau=m,\ d(A)=d}}
    k_\tau W_{\btau},
\end{equation}
with $\btau$ running over isomorphism classes of decorated wall types. We will
see momentarily that this sum is finite. The contact order in this formula is
expressed in the chosen endpoint cell. This sum is finite. Indeed, by
Lemma~\ref{Lem: support of maps with at most one puncture}, the total curve
class belongs to the cone in \eqref{Eqn: effective cone}. In a fixed endpoint
cell the total contact order determines the intersections of $A$ with the
components of $X_0$, by the degree-contact relation for punctured maps
\cite[\S2.2]{ACGS2}. The table in \S2 shows that the common kernel of
intersection with $P,D_1,D_2,D_3$ on the group generated by $\NE(Y)$ is
$\ZZ\beta$, with $\beta$ as in \eqref{Eqn: specialization of a line}. Since
$H\cdot\beta=1$, the degree $d$ and contact order determine $A$. Finiteness for
fixed curve class, as in \cite[Prop.\,3.11]{CanScat}, then makes each sum in
\eqref{Eqn: aggregate wall invariant} finite. Consistency makes the sum
independent of the general point and of the choice of endpoint cell when $-m$
lies on a common face. The factor $k_\tau/|m|$ is part of the aggregation
convention.

Replacing $q^A$ by $t^{d(A)}$ in the wall expansion, with the tangent monomials
retained, gives
\begin{equation}
\label{Eqn: aggregate canonical wall function}
  f=\exp\Big(
    \sum_{d\geq0,\ m\in\ZZ^2\setminus\{0\}}
    |m|\,W_{d,m}\,t^d\vartheta_{-m}\Big)
  \in\QQ\lfor x,y,z,t\rfor/(xyz-t).
\end{equation}
In particular, the terms in $\log f$ with contact orders $-k\be_1$ and $k\be_1$
are
\[
  kW_{d,-k\be_1}t^dx^k,\qquad
  kW_{d,k\be_1}t^d(yz)^k
  =kW_{d,k\be_1}t^{d+k}x^{-k},
\]
as used in Theorem~\ref{Thm: logarithmic period coefficient}.

We also fix the closed invariants to be compared with these wall invariants. For
$d>0$, let
\[
  N_d=\deg\big[\overline{\scrM}_{0,0}(K_{\PP^2},d)\big]^{\virt}
\]
be the genus-zero Gromov--Witten invariant of local $\PP^2$. The moduli space is
proper since every positive-degree stable map has image in the zero section.
The specialization of its curve class is $d\beta$.

A realizable decorated closed type $\btau=(\tau,\bA)$ has a genus-zero tree 
without
legs and total curve class $d\beta$. It is rigid if its tropical maps on $B$ have 
no
moduli, or equivalently if $\tau$ is a ray mapping nontrivially to the tropical
base. Let $p_\tau:\Lambda_\tau\to\ZZ$ be the induced homomorphism on lattices and 
put
\[
  m_\tau=[\ZZ:p_\tau(\Lambda_\tau)].
\]
Writing $b_0$ for the standard log point at $0\in\AA^1$, the logarithmic
decomposition formula \cite{ACGS1}, recalled in \cite[\S5.4]{ACGS2}, gives
\begin{equation}
\label{Eqn: closed decomposition}
  N_d=\sum_\btau \frac{m_\tau}{|\Aut(\btau)|}
  \deg[\scrM(X_0/b_0,\btau)]^{\virt}.
\end{equation}
The sum is over isomorphism classes of rigid decorated types $\btau$ of total
curve class $\sum_v \bA(v)=d\beta$.

For every rigid decorated closed type $\btau$ occurring in
\eqref{Eqn: closed decomposition}, the natural identification
\[
  \scrM(X_0/b_0,\,\btau)\simeq\scrM(X,\,\btau)
\]
from \cite[Prop.\,5.13]{ACGS2} is compatible with virtual fundamental classes,
as in \cite[proof of Prop.\,3.18]{CanScat}. Thus the sum decomposition
\eqref{Eqn: closed decomposition} of $N_d$ holds with $[\scrM(X,\btau)]^\virt$
in place of $[\scrM(X_0/b_0,\btau)]^\virt$.

The following subsections relate the types and multiplicities in
\eqref{Eqn: closed decomposition} to those in \eqref{Eqn: aggregate wall
invariant}.


\subsection{Contact orders and vanishing restrictions}
\label{SS: contact orders and vanishing}

We derive further restrictions on wall types and rigid closed types from
generalized balancing and the splitting formula with toric gluing strata.

For use in the generalized balancing condition, we associate to a compact
curve class $C$ its \emph{curve contact order}
\begin{equation}
\label{Eqn: curve contact order}
  u_C=(C\cdot P)\check P+\sum_{i=1}^3(C\cdot D_i)\check D_i.
\end{equation}
Here $\check P$ and $\check D_i$ are the primitive generators of the
corresponding rays of $\Sigma(X)$. On the cone over $B_0$,
the coordinates of \S2 give
\[
  \check P=(0,0,1),\quad
  \check D_1=(1,0,1),\quad
  \check D_2=(0,1,1),\quad
  \check D_3=(-1,-1,1).
\]
The last coordinate of $u_C$ vanishes, since $P+D_1+D_2+D_3$ is principal in a
neighborhood of $V(s+1)$. Omitting this coordinate, the
intersection table of \S2 gives
\begin{equation}
\label{Eqn: curve contact orders}
\begin{aligned}
  u_\ell&=u_{h_1}=u_\beta=0,
  &u_{h_2}&=(0,-3),
  &u_{h_3}&=(2,1),\\
  u_{e_1}&=(-2,-1),
  &u_{e_2}&=(-1,1),
  &u_{e_3}&=(1,2),\\
  u_f&=u_{f_2}=(0,-1),
  &u_{f_1}&=(-1,0),
  &u_{f_3}&=(1,1).
\end{aligned}
\end{equation}

For the balancing condition, we make the following observation.

\begin{lemma}
\label{Lem: boundary tree contact}
Let $\btau=(\tau,\bA)$ be the decorated type of a punctured logarithmic map,
with underlying graph $G$ and exactly one leg $L$. Suppose all vertices
and bounded edges map to $B_0$, and let $G_\partial\subseteq G$ be the subgraph
mapping to $\partial B_0$. Then
\begin{equation}
\label{Eqn: boundary tree contact}
  u_\tau=u_L=\sum_{v\in V(G)}u_{\bA(v)}
    =\sum_{v\in V(G_\partial)}u_{\bA(v)}.
\end{equation}
\end{lemma}

\begin{proof}
Summing generalized balancing \cite[\S2.2]{ACGS2} over all vertices cancels the
contributions of the bounded edges, leaving $u_L=\sum_{v\in V(G)}u_{\bA(v)}$.
The curve contact orders at vertices mapping to $\Int B_0$ vanish, so only
vertices of $G_\partial$ contribute.
\end{proof}

The same argument for a decorated type with no legs and all decorating cells in
$B_0$ gives $\sum_v u_{\bA(v)}=0$. Contracted edges cause no change to either
identity, since their contact orders vanish. Thus a connected decorated subgraph
carrying no leg and with all attaching edges of contact order zero has total
curve contact order zero. The classes $\widetilde h_i=\beta$ from \S2 satisfy
this condition.

We next apply the splitting formula of \cite{Yixian}, reviewed for the relevant
special case in \cite[App.\,A]{CanScat}, to obtain further restrictions
on the types to be considered. Call an edge $E$ a \emph{boundary attachment} if
exactly one of its adjacent vertices maps to $\partial B$.

\begin{lemma}
\label{Lem: vanishing from splitting}
A decorated wall type or rigid closed type $\btau$ with nonzero invariant has at
most one boundary attachment.
\end{lemma}

\begin{proof}
Suppose there are at least two boundary attachments. Call a maximal connected
subgraph spanned by vertices in $\Int B_0$ an interior subtree. Ordinary
balancing excludes an interior subtree with just one attachment unless it
carries the original leg. Since there is at most one original leg, some interior
subtree has at least two attachments. Split along all boundary attachments.

All contact orders of $\btau$ are parallel to $B_0$. Since the vertex classes
have zero intersection with $P_\pm$, generalized balancing shows that this
remains true under every realizable generization of a split type occurring in
Wu's formula. After splitting, each tropical piece $\bomega$ therefore has
constant normal coordinate. Every leg $L$ of a boundary piece $\bomega$ is
pointing inwards and hence fulfills $u_L(\varphi_P)>0$ for $\varphi_P$ the
piecewise affine function on $B$ defined by $P$. Thus the degree--contact
relation for $\bomega=(\omega,A')$ shows $P\cdot A'>0$. If the common normal
coordinate were nonzero, then for sufficiently small displacements, the
corresponding stable map has image in $P\cap P_+$ or $P\cap P_-$, hence has
class $A'=d\ell$. Thus $P\cdot A'=-3d\leq0$, contrary to $P\cdot A'>0$. Hence
the boundary pieces remain in $B_0$.

The gluing strata are toric strata of $P$, so Wu's splitting formula applies.
Choose displacement vectors with distinct normal components at two attachments
of the chosen interior subtree. The boundary endpoints have normal coordinate
zero, while the interior piece has a single normal coordinate. The displaced
tropical matching conditions are therefore incompatible, and every term in the
splitting formula vanishes.
\end{proof}

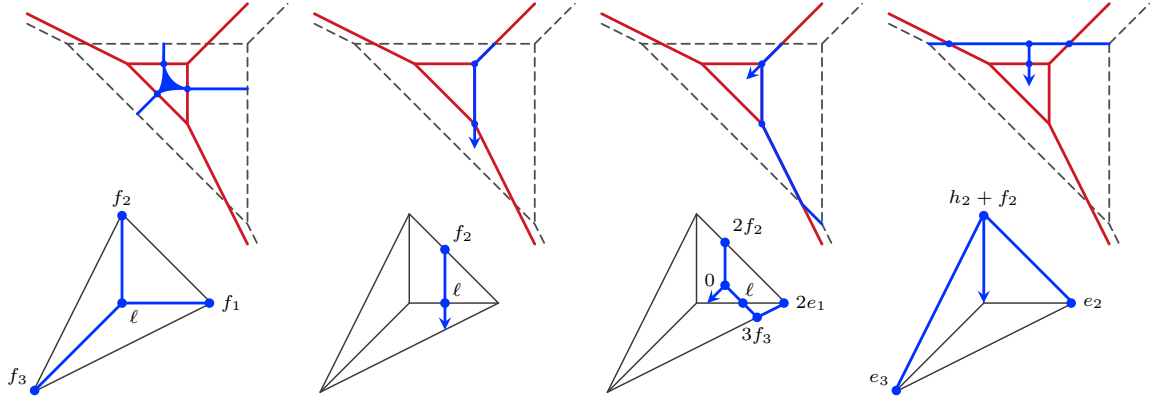
\begin{figure}[tbp]
\centering \input{TypesExamples.tikz} \caption{Four configurations of images of
punctured maps and their tropicalizations, denoted (a)--(d) from left to right.
Arrows indicate punctures.} \label{Fig: types of punctured maps}
\end{figure}

Figure~\ref{Fig: types of punctured maps} illustrates the restriction. In (a), a
line in the $\PP^2$ slice of $P$, together with one fiber in each adjacent
blown-up Hirzebruch surface, has total class $\beta$. Its tropicalization is a
rigid closed type with three boundary attachments, so its contribution vanishes
by Lemma~\ref{Lem: vanishing from splitting} although it deforms to a line in a
general fiber of $X\to \AA^1$. The wall type in (c) has two boundary
attachments, and its contribution vanishes by the same argument. The wall type
in (b) has one boundary attachment. In (d), all vertices and bounded edges lie
on $\partial B_0$, and the outgoing contact order is given by Lemma~\ref{Lem:
boundary tree contact}.

In the next subsection we further restrict wall types and closed types
contributing to $W_{d,-k\be_i}$ and to $N_d$, respectively.


\subsection{Contributing wall types}
\label{SS: shape of wall types}

We orient the graph of a wall type towards its outgoing leg.

\begin{proposition}
\label{Prop: shape of wall types}
Let $\btau=(\tau,\bA)$ be a decorated wall type with $W_\btau\neq 0$ and
underlying tree $G$. Then the part $G_\partial$ mapping to $\partial B_0$ is a
connected tree with a unique last vertex $v_\partial$. This vertex maps to the
interior of a boundary edge, all other vertices of $G_\partial$ to boundary
vertices. Every edge of $G_\partial$ has contact order in $\bigcup_{i=1}^3\ZZ
u_{e_i}$.

Starting at $v_\partial$, the remainder is a straight chain of at most two
bounded edges followed by the outgoing leg, all with contact order $u_\tau$. Its
vertices other than $v_\partial$ lie on edges of $B_0$ not contained in
$\partial B_0$.
\end{proposition}

\begin{proof}
The wall-type dimension conditions exclude nontrivial deformations fixing the
outgoing evaluation. In particular, every bounded edge has nonzero contact
order, since otherwise its length could vary independently. Since $u_\tau\neq0$,
Lemma~\ref{Lem: boundary tree contact} shows that $G_\partial$ is nonempty, and
Lemma~\ref{Lem: vanishing from splitting} shows that it is connected. Ordinary
balancing excludes interior leaves, so the remainder is a straight chain of
contact order $u_\tau$, starting at the last vertex $v_\partial$ of
$G_\partial$.

The outgoing leg sweeps out a two-dimensional set, so its supporting line must
vary. Hence $v_\partial$ maps to the interior of a boundary edge, and the chain
avoids the central vertex. Any other vertex of $G_\partial$ mapping to the
interior of a boundary edge, or any vertex of the chain mapping to the interior
of a triangle, could move while fixing the outgoing evaluation. Thus all other
vertices of $G_\partial$ map to boundary vertices, and the chain is subdivided
only at interior edges. A straight line avoiding the central vertex crosses at
most two such edges. Finally, the primitive tangent vectors of the boundary
edges are the $u_{e_i}$, up to sign.
\end{proof}

\begin{remark}
\label{Rem: corner trees}
At a vertex $v$ mapping to a boundary vertex, let $E_v$ be the
outgoing edge and $E_1,\ldots,E_s$ the incoming edges. Generalized balancing
reads
\begin{equation}
\label{Eqn: corner balancing}
  u_{E_v}-\sum_{j=1}^s u_{E_j}=u_{\bA(v)}.
\end{equation}
For example, at $(1,0)$ the primitive boundary directions pointing away from the
vertex are $u_{e_2}=(-1,1)$ and $u_{e_1}=(-2,-1)$. Writing $a,b$ for the total
incident weights in these directions gives
\[
  u_{\bA(v)}=a\,u_{e_2}+b\,u_{e_1}=(-a-2b,a-b).
\]
For $a=b$ this becomes $3a\,u_{f_1}$. The other two corners are treated by the
same balancing rule with their boundary directions.

Although $G_\partial$ has no edges of contact order zero, such edges may occur
in the actual decorated type $\btau'=(\tau',\bA')$ of a punctured map marked by
$\btau$. Let $H$ be a connected subtree of its graph carrying no leg and with
all attaching edges of contact order zero. Summing the degree--contact relation
\cite[Prop.\,2.29]{ACGS2} over its vertices cancels the internal edge
contributions and gives $\sum_{v'\in V(H)}\bA'(v')\cdot D=0$ for
$D=P,D_1,D_2,D_3$. By the intersection table in \S2, the total class of $H$
pushes forward to $d\beta$ for some $d\geq0$. If $H$ is marked by a single
vertex, this excludes positive multiples of $\ell$ as total curve class,
although $u_\ell=0$.

For a subtree marked by several vertices of $\btau$, the individual curve
contact orders need not vanish. For example, the representatives of $\beta$ from
\S2 give
\[
  u_{h_1}=u_{h_2}+u_{e_2}+u_{e_3}=u_{h_3}+u_{e_1}=0.
\]
Thus the contact order constraint allows closed types to be connected to a wall
type at any pair of vertices with the same image through an edge of contact
order zero without changing the outgoing vector.
\end{remark}

For the wall invariants entering $W_d(\lambda)$, the straight continuation after
$v_\partial$ does not occur.

\begin{lemma}
\label{Lem: relevant wall root}
Let $k>0$ and let $\Delta\subset B_0$ be a maximal triangle whose tangent cone
at the central vertex contains $\be_i$. If a decorated wall type with nonzero
invariant has contact order $u_\tau=-k\be_i$ and its wall contains a general point
of $\Delta$, then its root is $v_\partial$.
\end{lemma}

\begin{proof}
Choose a tropical representative whose outgoing leg passes through a general
point $p\in\Int\Delta$. Following the line backwards from $p$ in direction
$\be_i$ stays in $\Delta$ until it reaches the interior of its boundary edge. It
meets no interior edge of $B_0$. By Proposition~\ref{Prop: shape of wall types},
any vertex between $v_\partial$ and the outgoing leg would lie on such an
interior edge. Hence there is no bounded continuation, and the outgoing leg is
attached to $v_\partial$.
\end{proof}

Consequently, every contributing type in $W_{d,-k\be_i}$, with the endpoint
convention of \eqref{Eqn: aggregate wall invariant}, has all its vertices and
bounded edges on $\partial B_0$. Its last vertex lies in the interior of a
boundary edge and carries the outgoing leg.


\subsection{Point constraint interpretation of some wall invariants}
\label{SS: alternative wall invt interpretation}

For a tropical type $\eta$, let $Q_\eta$ denote its basic monoid and write
$\eta_\ZZ^\gp=\Lambda_\eta$ for the integral lattice of its parameter cone.

For concreteness, we consider outgoing contact order $-k\be_1=(-k,0)$, with $k>0$.
The cases $-k\be_2$ and $-k\be_3$ are analogous. Following \S2, we use local
defining equations $r,v,z$ for $P,D_1,D_2$, respectively. Their images in the
ghost sheaf $\ol\M_X$ extend uniquely to global sections denoted $\ol r,\ol
v,\ol z$ that vanish at the generic points of the other divisors. Note these can
also be viewed as integral piecewise linear functions on $\Sigma(X)$. Let
$\Delta\subset B_0$ be the triangle spanned by $(0,0),(1,0),(0,1)$. On the cone
over $\Delta$, a cone in $\Sigma(X)$, the restrictions of $(\ol v,\ol
z,\rho)$ are integral linear coordinates, with $\rho=\ol r+\ol v+\ol z$; the
$D_1$-ray is generated by $(1,0,1)$. Their restrictions to $\Delta$ give the
affine coordinates $(\ol v,\ol z)$ of the diagram, but they take value $0$
at $(0,0),(-1,-1)\in B_0$, hence do not define coordinates on the other
triangles. The essential point for the following discussion is that pairing with
$\ol v$ and $\ol z$ is enough to detect all relative contact orders that
specialize to the generic point of $P\cap D_1\cap D_2$.

Put $S_1=P\cap D_1\cap V(\sigma)\simeq\PP^1$ and let $p_1=S_1\cap D_2$ denote
the reduced intersection subscheme. The projection $P\cap D_1\to S_1$ pulls
$p_1$ back to the reduced Cartier divisor $P\cap D_1\cap D_2$, a $\PP^1$
intersecting $V(s+1)$ in $p_1$.

Let $\btau=(\tau,\bA)$ be a contributing wall type whose root vertex maps to the
interior of the boundary edge joining $(1,0)$ to $(0,1)$. We rarely refer to the
undecorated type $\tau$; when we do, the distinction from the deformation
parameter for $X\to\AA^1$ should be clear from context. Specializing this image
to $(1,0)$ determines a face $\delta$ of the two-dimensional cone $\tau$.
Contract all bounded edges whose lengths vanish on this face and sum the
push-forwards of the vertex decorations on each connected contracted subgraph.
The remaining contact orders of edges are carried along the corresponding face
maps. This gives a realizable decorated type $\bdelta$ with underlying cone the
ray $\delta$ and a contraction $c:\btau\to\bdelta$. Its outgoing contact order
remains $(-k,0)$. Fix such a decorated type $\bdelta$.

Composition with the contraction $c:\btau\to\bdelta$ turns a $\btau$-marking
into a $\bdelta$-marking. It therefore induces morphisms
\[
j_c:\scrM(X,\btau)\to\scrM(X,\bdelta),\quad
j_c^\ev: \frakM^\ev(\cX,\btau)\to \frakM^\ev(\cX,\bdelta).
\]
Here the evaluation stacks retain evaluation at the outgoing puncture and $\cX$
is the Artin fan of $X$. For $\btau$, this evaluation takes values in $P\cap
D_1\cap D_2\simeq \PP^1$, for $\bdelta$ in $P\cap D_1\simeq \PP^1\times\PP^1$.
Composing with the projection $P\cap D_1\to S_1$ defines
\[
\varepsilon\colon \frakM^\ev(\cX,\bdelta)\to S_1\simeq \PP^1.
\]
We use the same notation for its composite with $\scrM(X,\bdelta)\to
\frakM^\ev(\cX,\bdelta)$.

For a realizable decorated type $\boldsymbol\eta=(\eta,\mathbf A)$ with a given
contraction to $\bdelta$, let $v_{\mathrm{root}}$ be its root vertex and let
$h_{v_{\mathrm{root}}}:\eta\to\bsigma(v_{\mathrm{root}})$ be evaluation at this
vertex in the universal tropical family. Define $q_\eta=
h_{v_{\mathrm{root}}}^*\ol z\in Q_\eta$. This definition is compatible with
contractions, and $q_\delta=0$ since $\ol z$ vanishes on the ray in $\Sigma(X)$
corresponding to $D_1$. In particular, $q_\eta$ vanishes on the ray
$\delta\subset \eta$.

Let $\cR_{\bdelta,p_1}$ be the set of isomorphism classes, over the fixed
decorated target $\bdelta$, of contractions $c:\btau\to\bdelta$ with $\tau$
realizable, $\delta$ a codimension-one face of $\tau$, and $q_\tau\neq0$. Since
$q_\tau|_\delta=q_\delta=0$ and $q_\tau\neq0$, the codimension-one face $\delta$
equals $\tau\cap q_\tau^\perp$. Regarding $q_\tau$ as an integral linear
function on $\tau$, define $r_c>0$ by $q_\tau(\tau_\ZZ^\gp)=r_c\ZZ$. Write
$\Aut(c)=\{\alpha\in\Aut(\btau)\mid c\circ\alpha=c\}$. We impose no planarity
condition on these refinements of $\bdelta$; the divisor identity below is on
the full evaluation stack.

\begin{lemma}
\label{Lem: wall incidence divisor}
The following identities hold in the Chow groups of
$\frakM^\ev(\cX,\bdelta)$ and $\scrM(X,\bdelta)$, respectively.
\begin{equation}
\label{Eqn: wall incidence cycles}
\begin{aligned}
 \varepsilon^*[p_1]\cap[\frakM^\ev(\cX,\bdelta)]
 &=
 \sum_{[c:\btau\to\bdelta]\in\cR_{\bdelta,p_1}}
 \frac{r_c}{|\Aut(c)|}(j_c^\ev)_*
 [\frakM^\ev(\cX,\btau)],\\
 \varepsilon^*[p_1]\cap[\scrM(X,\bdelta)]^{\virt}
 &=
 \sum_{[c:\btau\to\bdelta]\in\cR_{\bdelta,p_1}}
 \frac{r_c}{|\Aut(c)|}(j_c)_*
 [\scrM(X,\btau)]^{\virt}.
\end{aligned}
\end{equation}
\end{lemma}

\begin{proof}
The morphisms forgetting curve-class decorations are strict and \'etale, so the
local descriptions and dimension formulas for undecorated stacks from
\cite{ACGS2} apply also to the decorated stacks. For a realizable type $\eta$
occurring here, \cite[Prop.\,3.30]{ACGS2} says that $\frakM(\cX,
\boldsymbol\eta)$ is reduced and pure-dimensional of dimension $-2-\dim\eta$.
Retaining evaluation at the unique puncture adds $\dim X=4$ by
\cite[Def.\,5.14]{ACGS2}. Thus $\frakM^\ev(\cX,\boldsymbol\eta)$ is reduced and
pure of dimension $2-\dim\eta$. The obstruction theory over the evaluation stack
has virtual relative dimension $c_1(\Theta_X)\cdot A+4(1-g-|\mathbf S|)=0$ by
\cite[(4.18)]{ACGS2}, since $X$ is logarithmically Calabi--Yau, $g=0$ and
$\mathbf S=\{L_{\mathrm{out}}\}$. Consequently, the evaluation stack and virtual
class for $\bdelta$ have dimension one, and those for the refinements in
$\cR_{\bdelta,p_1}$ have dimension zero.

Consider a geometric point of $\frakM^\ev(\cX,\bdelta)$ with actual
decorated type $\boldsymbol\eta$. Its marking gives a contraction 
$\boldsymbol\eta\to\bdelta$
and hence a localization homomorphism $\chi:Q_\eta\to Q_\delta\simeq \NN$. The
marking ideal of \cite[Def.\,3.20]{ACGS2} imposes the target strata and the
nodes prescribed by $\delta$, together with the basic-monoid localization
condition. Since $\delta$ is realizable, \cite[Prop.\,3.24]{ACGS2} identifies
the union of the marking and puncturing ideals with
$I_\delta=\chi^{-1}(Q_\delta\setminus\{0\})$.

The complement $F=\chi^{-1}(0)$ is a facet of $Q_\eta$. By
\cite[Rem.\,3.27]{ACGS2}, a smooth local chart of the $\bdelta$-marked stack has
monoidal factor
\[
 \Spec\bigl(\CC[Q_\eta]/(\mathrm u^a\mid a\in I_\delta)\bigr)
 =\Spec\CC[F].
\]
Here $\mathrm u^a$ denotes the monomial associated to $a$. The passage to the 
evaluation stack is strict and smooth, so the same
description applies there.

We compute the fiber equation at a point over $p_1$. The toric coordinate $z$ on
$P$ restricts to a local parameter on $S_1$ at $p_1$. By the definition of
$q_\eta$ and the computation of the outgoing contact order, the pull-back of
$\ol z$ to the outgoing puncture is $(q_\eta,0)$. The evaluation stack records a
lift of the underlying puncture evaluation to $X$, and consequently
$\varepsilon^*(z|_{S_1})$ is a unit times $\mathrm u^{q_\eta}$ in the above
chart.

Compatibility with the contraction gives $\chi(q_\eta)=q_\delta=0$, hence
$q_\eta\in F$. Moreover, $q_\eta\neq0$ at a point over $p_1$, since otherwise
the fiber equation would be a unit. Since $F$ is integral, multiplication by
$\mathrm u^{q_\eta}$ in $\CC[F]$ is injective. The smooth local description
therefore shows that $\varepsilon^{-1}(p_1)$ is an effective Cartier divisor.

The support of $V(\mathrm u^{q_\eta})\subset\Spec\CC[F]$ is the union of the
toric divisors defined by $(\mathrm u^a\mid a\in F\setminus H)$, where $H$ runs
over the facets of $F$ not containing $q_\eta$. The dual face $\tau=\eta\cap
H^\perp$ gives a factorization $\boldsymbol\eta\to\btau\xrightarrow{c}\bdelta$
with $\tau$ realizable, $\dim\tau-\dim\delta=1$ and $q_\tau\neq0$. Thus
$[c]\in\cR_{\bdelta,p_1}$. By \cite[Prop.\,3.24]{ACGS2}, the union of the
$\tau$-marking and puncturing ideals induces the ideal $(\mathrm u^a\mid a\in
F\setminus H)$ in $\CC[F]$. Hence these toric divisors are the local images of
the marking maps $j_c^\ev$, proving that the marking maps cover the irreducible
components of $\varepsilon^{-1}(p_1)$.

At a generic point of such a component, take $\eta=\tau$ and write
$F=Q_\tau\cap\delta^\perp=\NN p$ with $p$ primitive. Then $q_\tau=r_cp$ by the
definition of $r_c$. The fiber equation is a unit times $(\mathrm u^p)^{r_c}$,
while the $\btau$-marked locus is defined by $(\mathrm u^p)$. Thus the component
has multiplicity $r_c$.

The marking map $j_c^\ev$ is finite of generic degree $|\Aut(c)|$ onto its
image, since the compatible $\btau$-markings of a $\bdelta$-marked object of
actual type $\tau$ form an $\Aut(c)$-torsor, as in \cite[Rem.\,3.31 and the
proof of Thm.\,5.25]{ACGS2}. Dividing by this degree and summing the
multiplicities proves the first identity.

For the virtual identity, apply the same divisor computation with curve
decorations forgotten. For a fixed underlying contraction $c_0:\tau\to\delta$,
\cite[Prop.\,5.21]{ACGS2} identifies the pull-back of $\frakM^\ev(\cX,\tau)$
along $\scrM(X,\bdelta)\to\frakM^\ev(\cX,\delta)$ with the disjoint union of
$\scrM(X,\btau)$ over all compatible source decorations. The obstruction
theories agree under this pull-back. Virtual pull-back commutes with proper
push-forward and intersection with Cartier divisor classes, so it gives the sum
over these decorations. The group $\Aut(c_0)$ acts on the decorations, with
stabilizer $\Aut(c)$ for a decorated contraction $c$. Each orbit therefore has
$|\Aut(c_0)|/|\Aut(c)|$ elements. Grouping by these orbits turns the coefficient
$r_c/|\Aut(c_0)|$ into $r_c/|\Aut(c)|$, proving the second identity.
\end{proof}

Let $\cT_{\bdelta,p_1}$ be the set of isomorphism classes of decorated wall
types whose specialization along the boundary edge joining $(1,0)$ to $(0,1)$
contracts to a decorated type isomorphic to $\bdelta$. For the generator $b_0$
of $\delta_\ZZ\simeq\NN$, put $m_\delta=(\rho\circ
h_{v_{\mathrm{root}}})(b_0)>0$, where $h_{v_{\mathrm{root}}}$ is evaluation at
the root vertex of $\delta$. Equivalently,
$h_{v_{\mathrm{root}}}(b_0)=m_\delta\check D_1$.

\begin{proposition}
\label{Prop: wall degree comparison}
The moduli space $\scrM(X,\bdelta)$ is proper. For a general point $q\in S_1$,
\begin{equation}
\label{Eqn: wall degree comparison}
 \frac1k\sum_{\btau\in\cT_{\bdelta,p_1}}
 k_\tau W_{\btau}
 =
 \frac{m_\delta}{|\Aut(\bdelta)|}
 \deg\bigl(\varepsilon^*[q]\cap[\scrM(X,\bdelta)]^{\virt}\bigr).
\end{equation}
\end{proposition}

\begin{proof}
Properness of the moduli space has already been proved in Lemma~\ref{Lem: support 
of maps with at most one puncture}.

Consider a term of nonzero degree in the second identity of \eqref{Eqn: wall
incidence cycles}, corresponding to $c:\btau\to\bdelta$. The moduli space
$\scrM(X,\btau)$ is nonempty. The actual type of a map in it refines $\tau$ and
satisfies generalized balancing, which descends under contraction. Thus $\tau$
is balanced in the sense of \cite[Def.\,2.2]{CanScat}.

Let $v_{\mathrm{root}}$ be its root vertex and let $b_0$ be the generator of
$\delta_\ZZ\simeq\NN$. Since $\delta_\ZZ^\gp\subset \tau_\ZZ^\gp$ is saturated,
we can extend $b_0$ to a basis $(b_0,b_1)$ and then $q_\tau(b_1)=r_c$. In the
lattice of $\bsigma(L_{\mathrm{out}})$, the restriction of $\ol z$ vanishes on
$h_{v_{\mathrm{root}}}(b_0)$ and the full outgoing vector $u_\tau$, and takes
value $r_c$ on $h_{v_{\mathrm{root}}}(b_1)$. Moreover,
$\rho(h_{v_{\mathrm{root}}}(b_0))=m_\delta>0$, whereas $\rho(u_\tau)=0$ and
$u_\tau\neq0$. These three vectors are therefore linearly independent. Hence the
outgoing-leg cone has three-dimensional image, and $\dim\tau=2$ makes $\tau$ a
wall type.

The support restriction of Lemma~\ref{Lem: support of maps with at most one
puncture} now puts its height-one tropicalization in $B_0$. Thus non-planar
refinements have zero degree in the virtual identity and can be ignored. By
Proposition~\ref{Prop: shape of wall types}, the root vertex $v_{\mathrm{root}}$
is either the last boundary vertex or has image on an interior edge. Since the
image of $v_{\mathrm{root}}$ specializes to $(1,0)$, the only possible interior
edge is the one joining the central vertex to $(1,0)$. But $\ol z$ vanishes on
the cone over this edge, contradicting $q_\tau\neq0$. Thus $v_{\mathrm{root}}$
is the last boundary vertex. Of the two boundary edges adjacent to $(1,0)$, only
the edge joining it to $(0,1)$ has positive $\ol z$ on its interior. Hence
$\btau\in\cT_{\bdelta,p_1}$.

The outgoing-leg image is now contained in the cone over $\Delta$, where we may
use the coordinates $(\ol v,\ol z,\rho)$ fixed above. Evaluation at the root
vertex sends $b_0$ to $m_\delta(1,0,1)$ and $b_1$ to $(a,r_c,b)$ for some
$a,b\in\ZZ$, since $q_\tau=h_{v_{\mathrm{root}}}^*\ol z$. Together with the
integral parameter along the outgoing leg, $b_0,b_1$ generate its integral
lattice. Their images $m_\delta(1,0,1)$, $(a,r_c,b)$ and $(-k,0,0)$ generate a
sublattice of $\ZZ^3$ of index $k\,m_\delta r_c$. Thus $k_\tau=k\,m_\delta r_c$.

It remains to group the contractions by their decorated source types. Fix
$\btau\in\cT_{\bdelta,p_1}$ and let $\bdelta'$ be its corner contraction.
Every automorphism of $\btau$ preserves the corner face $\tau\cap q_\tau^\perp$
and hence induces an automorphism of $\bdelta'$. The isomorphisms
$\bdelta'\to\bdelta$ form an $\Aut(\bdelta)$-torsor on which $\Aut(\btau)$ acts.
Its orbits are the isomorphism classes of contractions $c:\btau\to\bdelta$ over
the fixed target, and their stabilizers are $\Aut(c)$. Thus the sum of
$1/|\Aut(c)|$ over these contractions is $|\Aut(\bdelta)|/|\Aut(\btau)|$. The
integer $r_c$ is determined by $q_\tau$ and hence is independent of the chosen
isomorphism $\bdelta'\to\bdelta$.

\sloppy
Taking degrees in \eqref{Eqn: wall incidence cycles}, grouping by decorated
source types, and using $W_{\btau}=\deg[\scrM(X,\btau)]^\virt/|\Aut(\btau)|$ and
$k_\tau=k\,m_\delta r_c$ proves the formula with $q=p_1$. Finally, rational
equivalence of $p_1$ and $q$ on $S_1\simeq\PP^1$ gives
$\varepsilon^*[p_1]=\varepsilon^*[q]$, proving the proposition.
\end{proof}

\fussy The same argument with $D_3$ in place of $D_2$ expresses the general
point constraint as the sum for wall types whose root vertices map to the
interior of the other boundary edge adjacent to $(1,0)$.


\subsection{Splitting at a distinguished vertex}
\label{SS: cutting at the last vertex}

We treat two kinds of decorated types $\bdelta=(\delta,\bA)$ simultaneously.
In the punctured case, $\bdelta$ is obtained in \S\ref{SS: alternative wall
invt interpretation} by specializing the image of the root vertex $v_0$ to a
boundary vertex. Its outgoing contact order is $-k\be_i$, with $k>0$, and
$v_0$ denotes the root vertex. In the closed case, $\bdelta$ is a rigid type
contributing to \eqref{Eqn: closed decomposition}, with total class $d\beta$,
$d>0$, and we choose a vertex $v_0$ with $d(\bA(v_0))>0$. Write $n=1$ and
$n=0$, respectively, for the number of original marked points.

In both cases every vertex maps to a boundary vertex of $B_0$, and every
bounded edge maps onto a boundary edge. For the punctured case this follows
from Proposition~\ref{Prop: shape of wall types} and Lemma~\ref{Lem: relevant
wall root}. In the closed case, Lemma~\ref{Lem: support of maps with at most
one puncture} together with classical balancing at vertices mapping to the
interior of $B_0$ forces at least two boundary attachments, contrary to
Lemma~\ref{Lem: vanishing from splitting}. Thus all vertices lie on $\partial
B_0$ and rigidity forces them to be boundary vertices.

Let $D_i$ be the component given by the strata decoration $\bsigma(v_0)$ at
$v_0$. In the punctured case choose a general point $q\in S_i=P\cap D_i\cap
V(s+1)$ and retain the notation $\varepsilon$ for the outgoing evaluation
composed with the projection $D_i\to S_i$. Put
\begin{equation}
\label{Eqn: point constrained delta contribution}
  J_{\bdelta}(q)
  =\frac{m_\delta}{|\Aut(\bdelta)|}
    \deg\bigl(\varepsilon^*[q]\cap
      [\scrM(X,\bdelta)]^\virt\bigr).
\end{equation}
By Proposition~\ref{Prop: wall degree comparison}, this is the sum of the
contributions $k_\tau W_{\btau}/k$ specializing to $\bdelta$ along the chosen
boundary edge. For a closed type put
\begin{equation}
\label{Eqn: closed delta contribution}
  N_{\bdelta}
  =\frac{m_\delta}{|\Aut(\bdelta)|}
    \deg[\scrM(X,\bdelta)]^\virt
\end{equation}
for the summand in \eqref{Eqn: closed decomposition} given by $\delta$ and the 
subsequent discussion on the relative-absolute comparison. Here, in
either case, $m_\delta=\big|\coker (\Sigma(\pi)\circ h_{v_0})\big|$, for
$h_{v_0}:\delta_\ZZ^\gp\to \bsigma(v_0)^\gp_\ZZ$ given by evaluation at $v_0$.
The properness needed for the degrees follows from Lemma~\ref{Lem: support of
maps with at most one puncture}.

Let $E_1,\ldots,E_a$ be the bounded edges incident to $v_0$, oriented away from
$v_0$. Splitting $\bdelta$ along these edges gives a rigid one-vertex type
$\bdelta_0$ with class $\bA(v_0)$ and, when $n=1$, the original outgoing
puncture, and connected rigid branch types $\bdelta_1,\ldots,\bdelta_a$. The two
legs replacing $E_j$ have contact orders $u_{E_j}$ on $\bdelta_0$ and $-u_{E_j}$
on $\bdelta_j$, and are labeled by $E_j$. Note $u_{E_j}\neq 0$ for all $j$ by
rigidity. Since splitting preserves all vertices and their decorations, the
decorated version of \cite[Prop.\,5.17]{ACGS2} gives a Cartesian diagram
expressing compatibility of the maps $\scrM\to\frakM$ with splitting. Write
$v_j$ for the endpoint of $E_j$ of the $j$-th branch $\bdelta_j$, and $h_v$ for
evaluation at the vertex $v$ and its induced homomorphism on lattices.

Every branch has one marked point, so its maps have image in $V(s+1,\tau)$ by
Lemma~\ref{Lem: support of maps with at most one puncture}. The central maps
have image in $Y_i$. Indeed, the argument of that lemma, applied to the
resolution $X\to\ol X$, together with $\bA(v_0)\cdot P_\pm=0$, shows that a
central map not contained in $Y_i$ has class a positive multiple of $\ell$ or of
one of the exceptional classes $e_r$. Since the contact orders $u_{E_j}$ have
zero $P$-component, the degree--contact relation \cite[Cor.\,2.30]{ACGS2} gives
\[
  P\cdot\bA(v_0)=
  \begin{cases}
    k,&n=1,\\
    0,&n=0.
  \end{cases}
\]
The intersection table of \S2 excludes both alternatives when $n=1$. When
$n=0$, it excludes multiples of $\ell$, while $d(\bA(v_0))>0$ excludes the
exceptional classes. Thus $\scrM(X,\bdelta_0)$ is proper by the same stable
reduction argument. The branch stacks need not be proper: covers of
exceptional curves can move along the unbounded strips.

Evaluation at the two markings labeled by $E_j$ takes values in the smooth 
surfaces
$X_{\bsigma(E_j)}$:
\[
  \ev_{0,j}\colon\scrM(X,\bdelta_0)\longrightarrow X_{\bsigma(E_j)},
  \qquad
  \ev_j\colon\scrM(X,\bdelta_j)\longrightarrow X_{\bsigma(E_j)}.
\]
When $n=1$, write $\varepsilon_0\colon\scrM(X,\bdelta_0)\to S_i$ for the
outgoing evaluation composed with $D_i\to S_i$. By \cite[Thm.\,5.1]{Gross}, the
splitting morphism factors over a finite representable morphism $\phi_\delta$
fitting into a commutative diagram of logarithmic stacks
\begin{equation}
\label{Eqn: root schematic gluing}
\begin{CD}
  \scrM(X,\bdelta) @>{\phi_\delta}>>
  \mathscr G_\delta @>{\iota}>>
    \displaystyle\prod_{j=0}^a\scrM(X,\bdelta_j)\\
  @. @VVV @VV{\prod_{j=1}^a(\ev_{0,j},\ev_j)}V\\
  @. \displaystyle\prod_{j=1}^a X_{\bsigma(E_j)} @>{\Delta}>>
    \displaystyle\prod_{j=1}^a
      \bigl(X_{\bsigma(E_j)}\times X_{\bsigma(E_j)}\bigr).
\end{CD}
\end{equation}
Here $\mathscr G_\delta$ is defined by Cartesianity of the square on the
underlying stacks and with the log structure making $\iota$ strict. When $n=1$,
the outgoing evaluation satisfies $\varepsilon= \varepsilon_0\circ
\operatorname{pr}_0\circ\iota \circ\phi_\delta$, where $\operatorname{pr}_0$
projects to the central factor $j=0$.

Consider the tropical gluing homomorphism
\begin{equation}
\label{Eqn: root tropical gluing map}
\begin{aligned}
  \Psi_\delta\colon
  \bigoplus_{j=0}^a\delta_{j,\ZZ}^\gp\oplus\ZZ^a
  &\longrightarrow\bigoplus_{j=1}^a\bsigma(E_j)_\ZZ^\gp,\\
  ((\xi_j),(t_j))
  &\longmapsto
    \bigl(h_{v_0}(\xi_0)-h_{v_j}(\xi_j)+t_ju_{E_j}\bigr)_{j=1}^a.
\end{aligned}
\end{equation}
The face inclusions place the vertex evaluations in the lattices 
$\bsigma(E_j)_\ZZ^\gp$, and
the $\ZZ^a$-summand records the domain edge lengths. Then
$\ker\Psi_\delta=\delta_\ZZ^\gp$ by the definition of basic monoids and
$\delta_\ZZ^\gp\simeq\ZZ$ by rigidity. Since the source and target have ranks
$2a+1$ and $2a$, respectively, $\coker \Psi_\delta$ is finite. Put
\begin{equation}
\label{Eqn: root splitting tropical multiplicity}
  \mu_\delta=|\coker\Psi_\delta|.
\end{equation}
For a closed type this multiplicity depends on the chosen vertex $v_0$ and we
write $\mu_{\delta,v_0}$.

We prove the splitting formula by attaching the branches one at a time. We first
establish flatness of evaluation at the leg used in each gluing.

\begin{lemma}
\label{Lem: single leg flatness at root}
Let $\boldsymbol\eta$ be a type occurring in the successive gluing, and let
$\mathbf L$ be a set of its legs containing a leg $L$ obtained by splitting an 
edge $E_j$. Then the evaluation
morphism
\[
  \ev_L\colon\frakM^{\ev(\mathbf L)}(\cX,\boldsymbol\eta)
  \longrightarrow X_{\bsigma_\eta(L)}
\]
is flat on an open substack containing the image of the structural morphism from
$\scrM(X,\boldsymbol\eta)$.
\end{lemma}

\begin{proof}
Splitting $\boldsymbol \eta$ along its edges incident to $v_0$ recovers the
central type and the attached branch types. The support argument above therefore
puts every such map in $V(s+1,\tau)$. By \cite[Rem.\,5.4]{Gross}, we may
restrict $\frakM(\cX,\eta)$ to the open substack whose strata correspond to
contractions of types occurring in $\scrM(X,\boldsymbol\eta)$. The retained
types lie in the cone over $B_0$.

Now $\eta$ is a ray and $\bsigma_\eta(L)$ has dimension two. For a retained
realizable refinement $\eta'\to\eta$, the dimension of $\bsigma_{\eta'}(L)$ is
two or three. In the latter case, $\eta$ is a proper face of $\eta'$, so
$\dim\eta'\geq2$. Hence, in either case,
\[
  \dim\eta'\geq \dim\eta + \dim\bsigma_{\eta'}(L)-\dim\bsigma_\eta(L).
\]
Apply \cite[Thm.\,5.3]{Gross} to evaluation at $L$ and base change along
$X_{\bsigma_\eta(L)}\to\cX_{\bsigma_\eta(L)}$. This proves flatness with only
the evaluation at $L$ lifted. Forgetting the remaining evaluation lifts is
smooth, since each $X_{\bsigma_\eta(L')}\to\cX_{\bsigma_\eta(L')}$ is smooth.
Thus evaluation at $L$ remains flat after adding these lifts.
\end{proof}

\begin{theorem}
\label{Thm: root virtual gluing}
The stack $\mathscr G_\delta$ is proper, and
\begin{equation}
\label{Eqn: root virtual splitting identity}
  (\phi_\delta)_*[\scrM(X,\bdelta)]^\virt
  =\mu_\delta\cdot\Delta^!\Bigl(\prod_{j=0}^a[\scrM(X,\bdelta_j)]^\virt\Bigr)
  \qquad\text{in }A_n(\mathscr G_\delta)_\QQ.
\end{equation}
For $a=0$, set $\mu_\delta=1$ and interpret $\Delta^!$ as the identity.
\end{theorem}

\begin{proof}
Glue the branches in order, and write $\widetilde{\bdelta}_j$ for the type
obtained after attaching the first $j$ branches. At the $j$-th step we obtain 
$\mathscr G_j$ as the fiber product and the finite representable morphism
\[
\phi_j: \scrM(X,\widetilde\bdelta_j)\lra \mathscr G_j.
\]
The tropical gluing homomorphism is
\[
\begin{aligned}
  \Phi_j\colon
  \widetilde\delta_{j-1,\ZZ}^\gp\oplus\delta_{j,\ZZ}^\gp\oplus\ZZ
  &\longrightarrow\bsigma(E_j)_\ZZ^\gp,\\
  (\xi,\xi_j,t)
  &\longmapsto h_{v_0}(\xi)-h_{v_j}(\xi_j)+tu_{E_j}.
\end{aligned}
\]
The two vertex evaluations have linearly independent images, since $u_{E_j}\neq
0$ implies that the endpoints lie on distinct rays. Thus $\Phi_j$ has finite
cokernel, and the gluing is tropically transverse. Put $\nu_j=|\coker\Phi_j|$.

Let $\Psi_j$ be \eqref{Eqn: root tropical gluing map} restricted to the first
$j$ branches. Since $\ker\Psi_{j-1}= \widetilde\delta_{j-1,\ZZ}^\gp$,
projection onto the first $j-1$ edge summands gives an exact sequence
\[
  0\longrightarrow\coker\Phi_j
  \longrightarrow\coker\Psi_j
  \longrightarrow\coker\Psi_{j-1}
  \longrightarrow0.
\]
As $\coker\Psi_0=0$, this proves
\begin{equation}
\label{Eqn: successive root gluing multiplicity}
  \mu_\delta=\prod_{j=1}^a\nu_j.
\end{equation}

On the auxiliary opens of Lemma~\ref{Lem: single leg flatness at root},
\cite[Thm.\,5.5]{Gross} gives a dominant enhancement morphism of degree $\nu_j$
on the evaluation stacks. By \cite[Thm.\,5.1]{Gross}, this dominant morphism is
finite and representable, hence surjective. The enhancement morphism $\phi_j$ of
Deligne-Mumford stacks is its base change. Compatibility of the obstruction
theories, by \cite[Thm.\,5.19]{ACGS2} and \cite[Rem.\,5.2]{Gross}, gives
\[
  (\phi_j)_*[\scrM(X,\widetilde{\bdelta}_j)]^\virt
  =\nu_j\cdot \Delta_j^!\Bigl(
    [\scrM(X,\widetilde{\bdelta}_{j-1})]^\virt
    \times[\scrM(X,\bdelta_j)]^\virt\Bigr),
\]
where $\Delta_j$ is the diagonal of $X_{\bsigma(E_j)}$.

Retain the evaluations at the legs still to be glued at each step. After the 
corresponding
base changes, the enhancement morphisms compose to $\phi_\delta$. Proper
pushforward commutes with refined pullback, the diagonal pullbacks compose to
$\Delta^!$, and their coefficients multiply to $\mu_\delta$. This proves
\eqref{Eqn: root virtual splitting identity}. The same argument makes
$\phi_\delta$ finite and surjective. Since its source is proper and its target
is separated and of finite type, $\mathscr G_\delta$ is proper.
\end{proof}

We next turn the diagonal product into numbers. Put
\begin{equation}
\label{Eqn: F_j,S,F}
  F_j=Y_i\cap X_{\bsigma(E_j)},\qquad
  S=\prod_{j=1}^a X_{\bsigma(E_j)},\qquad F=\prod_{j=1}^a F_j,
\end{equation}
and write $i_j\colon F_j\hookrightarrow X_{\bsigma(E_j)}$ and $i_F\colon
F\hookrightarrow S$ for the inclusions. Each $F_j$ is a Cartier divisor in
$X_{\bsigma(E_j)}$ isomorphic to $\PP^1$.

The branch evaluations $\ev_j$ are proper. Indeed, the only non-compact
$\scrM(X,\bdelta_j)$ have underlying stable maps multiple covers of curves of
the classes $e_1,e_2,e_3$ of the ruling of the surfaces $\PP^1\times\AA^1$ given
by the infinite strips in Figure~\ref{Fig: toric model}, and properness of
the evaluation to the double locus $X_{\bsigma(E_j)}$ is immediate in this case.
We may therefore define
\begin{equation}
\label{Eqn: branch evaluated class}
  B_j=(\ev_j)_*[\scrM(X,\bdelta_j)]^\virt\in A_1(X_{\bsigma(E_j)})_\QQ,
  \qquad
  b_j=\deg i_j^!B_j.
\end{equation}
The degree is taken in $A_0(F_j)$, on the proper curve $F_j$. It includes any
self-intersection contribution from components of $B_j$ supported on $F_j$, as
well as contributions from moving strip families.

Let $e_0\colon(\scrM(X,\bdelta_0))_{\mathrm{red}}\to F$ be the central
evaluation map. Define the numbers $R_{\bdelta_0}$, $R_{\bdelta_0}(q)$ computing
the respective contribution of $\scrM(X,\bdelta_0)$ by
\begin{equation}
\label{Eqn: central evaluated classes}
\begin{aligned}
  (e_0)_*[\scrM(X,\bdelta_0)]^\virt
    &=R_{\bdelta_0}[F] &&(n=0),\\
  (e_0)_*\bigl(\varepsilon_0^*[q]\cap
    [\scrM(X,\bdelta_0)]^\virt\bigr)
    &=R_{\bdelta_0}(q)[F] &&(n=1).
\end{aligned}
\end{equation}
These equalities uniquely define the numbers since $A_a(F)_\QQ=\QQ[F]$.

\begin{proposition}
\label{Prop: numerical splitting at a vertex}
With the preceding conventions, the closed and punctured contributions satisfy
\begin{align}
\label{Eqn: closed numerical reduction}
  N_{\bdelta}
  &=\frac{m_\delta\mu_\delta}{|\Aut(\bdelta)|}
    R_{\bdelta_0}\prod_{j=1}^a b_j,\\
\label{Eqn: root numerical reduction}
  J_{\bdelta}(q)
  &=\frac{m_\delta\mu_\delta}{|\Aut(\bdelta)|}
    R_{\bdelta_0}(q)\prod_{j=1}^a b_j.
\end{align}
Empty products have value one.
\end{proposition}

\begin{proof}
In the closed case, consider
\[
\begin{aligned}
  \gamma
    &=\Delta^!\bigl([\scrM(X,\bdelta_0)]^\virt
      \times\beta_{\mathrm b}\bigr),\\
  \beta_{\mathrm b}
    &=[\scrM(X,\bdelta_1)]^\virt\times\cdots
      \times[\scrM(X,\bdelta_a)]^\virt.
\end{aligned}
\]
The pullback of $\Delta$ along $i_F\times1$ is the graph $\Gamma_{i_F}=(1\times
i_F)\circ\Delta_F$, a regular embedding of the same codimension. Compatibility
of refined Gysin pullback with liftings of regular embeddings
\cite[Thm.\,6.2]{Fulton} therefore identifies $\Delta^!$ with $\Gamma_{i_F}^!$
on the common fiber product. Commuting proper pushforward with refined pullback
and using \eqref{Eqn: central evaluated classes} gives
\[
\begin{aligned}
  \deg\gamma
    &=\deg\Gamma_{i_F}^!\bigl((e_0)_*[\scrM(X,\bdelta_0)]^\virt
      \times(B_1\times\cdots\times B_a)\bigr)\\
    &=R_{\bdelta_0}\deg i_F^!(B_1\times\cdots\times B_a)
   =R_{\bdelta_0}\prod_{j=1}^a b_j.
\end{aligned}
\]
Here the degrees are taken on $F$. In the punctured case, insert
$\varepsilon_0^*[q]$ on the central factor. The same calculation, using the
projection formula, replaces $R_{\bdelta_0}$ by $R_{\bdelta_0}(q)$. The two
identities follow from Theorem~\ref{Thm: root virtual gluing} and the
definitions of $N_{\bdelta}$ and $J_{\bdelta}(q)$ in \eqref{Eqn: closed delta 
contribution} and \eqref{Eqn: point constrained delta contribution}.
\end{proof}

The factors in these formulas retain the labels $E_j$ on the legs obtained by 
splitting. Let
$\Aut(\bdelta,v_0)$ be the subgroup fixing the distinguished vertex and
$\Pi_{\bdelta,v_0}$ its permutation group on the edges $E_1,\ldots,E_a$. There is 
an exact
sequence
\[
  1\longrightarrow\prod_{j=1}^a\Aut(\bdelta_j)
  \longrightarrow\Aut(\bdelta,v_0)
  \longrightarrow\Pi_{\bdelta,v_0}\longrightarrow 1.
\]
The central one-vertex type has no automorphism fixing all its labeled legs. In
the punctured case the original leg already distinguishes $v_0$, so
$\Aut(\bdelta,v_0)=\Aut(\bdelta)$. Stable-map automorphisms are incorporated in
the virtual classes throughout.

For the comparison in \S5, sum over choices of the distinguished vertex. Recall
that $H$ is the pull-back to $X$ of the hyperplane class in $\PP^2$ and
$d(A)=A\cdot H$. Additivity of this $H$-degree and the orbit-stabilizer theorem
give
\[
  \sum_{[v_0]\in V(G)/\Aut(\bdelta)}
  \frac{d(\bA(v_0))}{|\Aut(\bdelta,v_0)|}
  =\frac{d}{|\Aut(\bdelta)|}.
\]
Terms with $d(\bA(v_0))=0$ disappear. Thus
\begin{equation}
\label{Eqn: vertex weighted closed splitting}
  dN_d
  =\sum_{[(\bdelta,v_0)]}
    \frac{m_\delta\mu_{\delta,v_0}\,d(\bA(v_0))}
         {|\Aut(\bdelta,v_0)|}
    R_{\bdelta_0}\prod_{j=1}^a b_j,
\end{equation}
where the sum is over isomorphism classes of pairs $(\bdelta,v_0)$ such that
$\bdelta$ is a rigid decorated closed type supported on $\partial B_0$, with
total class $d\beta$, and $d(\bA(v_0))>0$.

Finally, adding at $v_0$ an outgoing leg of contact order $-k\be_i$
decorated by the cone spanned by $\check P,\check D_i$, and replacing
$\bA(v_0)$ by $\bA(v_0)+kf_i$, gives a relevant punctured type $\bdelta(k)$.
All branch types and their evaluated classes remain the same. The vertex and
bounded-edge data determining the basic monoid are unchanged, so
\begin{equation}
\label{Eqn: closed punctured gluing data}
  m_{\delta(k)}=m_\delta,\qquad
  \mu_{\delta(k)}=\mu_{\delta,v_0},\qquad
  \Aut(\bdelta(k))\simeq\Aut(\bdelta,v_0).
\end{equation}
The branch numbers $b_j$ therefore agree term by term. Equations \eqref{Eqn:
root numerical reduction} and \eqref{Eqn: vertex weighted closed splitting} thus
reduce the comparison between $N_d$ and the $W_{d,-k\be_i}$ for $k>0$ to the
central numbers $R_{\bdelta_0(k)}(q)$ and $d(\bA(v_0))R_{\bdelta_0}$.
That these pairs exhaust the nonzero punctured contributions is proved in
Remark~\ref{Rem: exhaustion of punctured central factors} and the proof of
Corollary~\ref{Cor: wall versus closed comparison}.


\section{Localization on the rational ruled surfaces $Y_i$ and polynomiality}
\label{Sec: Localization on Y_i}

We compare the central invariants $R_{\bdelta_0}$ and $R_{\bdelta_0(k)}(q)$ of
\eqref{Eqn: central evaluated classes}, retaining the one-vertex decorated types
$\bdelta_0$, $\bdelta_0(k)$ and the notation of \S\ref{SS: cutting at the last
vertex}. We assume $d_0=d(A)>0$, where $A=\bA(v_0)$. In \S\ref{SS: polynomiality
and assembly}, the comparison will verify Assumption~\ref{Ass: wall
polynomiality} and identify the constant terms with closed invariants; the case
$d_0=0$ in positive total degree is also treated there.

Write $Y$ for the degeneration of $\PP^2$ in Proposition~\ref{Prop: Mumford
degeneration}, and $S_i=P\cap Y_i$ for the negative section of $Y_i$. Let
$a_i\colon Y_i\to S_i$ be the projection and $H_i\subset Y_i$ the strict
transform of the toric positive section of $\FF_3$ in the Mumford model $\PP_\Xi$.
Put $\breve S_i=S_i\setminus\bigcup_{j\ne i}D_j$ and fix a general point
$q\in\breve S_i$ throughout this section. Write $F_q=a_i^{-1}(q)$ and
$y_q=F_q\cap H_i$. As in \S\ref{SS: cutting at the last vertex}, $\varepsilon$
denotes outgoing evaluation followed by projection to $S_i$.

As before, let $\cX$ be the Artin fan of $X$, and put
\[
\begin{aligned}
  \scrM_0&=\scrM(X,\bdelta_0),
  &\frakM_0&=\frakM(\cX,\bdelta_0),\\
  \scrM_k&=\scrM(X,\bdelta_0(k)),
  &\frakM_k&=\frakM(\cX,\bdelta_0(k))\qquad(k>0).
\end{aligned}
\]
We also use $\scrM_0^+=\scrM(X,\bdelta_0^+)$ and
$\frakM_0^+=\frakM(\cX,\bdelta_0^+)$, where $\bdelta_0^+$ is obtained from
$\bdelta_0$ by adding a leg $L_p$ of contact order zero, corresponding to a new
marked point $p$. Using the factorization through $Y_i$ proved in \S\ref{SS: 
relative localization},
write
\begin{equation}
\label{Eqn: ev_p}
  \ev_p\colon\scrM_0^+\to Y_i,\qquad
  \varepsilon^+=a_i\circ\ev_p\colon\scrM_0^+\to S_i,\qquad
  \varepsilon\colon\scrM_k\to S_i\quad(k>0).
\end{equation}
Here $\ev_p$ evaluates at $p$, while $\varepsilon$ is outgoing evaluation followed
by projection to $S_i$. The stacks $\frakM_0$, $\frakM_k$, and $\frakM_0^+$
parametrize maps to $\cX$ with all marked points retained.

Use the one-parameter subgroup $T=\GG_m$ of the toric model fixing $\sigma$ and
$\tau$ and given by $-\be_i$ in the coordinates of $B_0$, hence by $(-1,0)$ for
$i=1$. It fixes $S_i$ and $H_i$ pointwise and acts along the fibers of $a_i$. We
write
\[
  \hbar=c_1^T(T_qF_q)
\]
for the equivariant first Chern class of the tangent line to $F_q$ at $q$.
Thus the tangent weights at $q$ and $y_q=F_q\cap H_i$ are $\hbar$ and
$-\hbar$, respectively. On a degree-$k$ cover of $F_q$ fully
ramified over these points, the tangent weight over $y_q$ is $-\hbar/k$.

We use the relative virtual localization theorem of Graber and Martin
\cite{GraberMartin}, since the Artin stack bases need not be smooth
(Remark~\ref{Rem: singular localization base}). For each moduli stack, replace
$T$ by a finite cover so that invariant points lift to its fixed stack, as in
\cite[Prop.\,4.11]{GraberMartin}. We keep the notation $T$ and identify the
rational equivariant Chow groups under this replacement.

An object of $\scrM_k^T$ over a scheme $W$ is an object of $\scrM_k(W)$ with
compatible $T$-actions on its logarithmic domain and base, making the map to $X$
equivariant and inducing the trivial action on $W$. The action on the base log
structure need not be trivial. The same description applies to our other
logarithmic moduli stacks.

\subsection{Normal excess and vanishing}
\label{SS: relative localization}

We first give equivariant expressions for the central numbers $R_{\bdelta_0}$
and $R_{\bdelta_0(k)}(q)$. Recall the fiber components $F_j=Y_i\cap
X_{\bsigma(E_j)}$ of $a_i\colon Y_i\to S_i$ introduced in \eqref{Eqn: F_j,S,F}.
Let $e_{0,j}=\operatorname{pr}_j\circ e_0$ be the corresponding evaluation map
to $F_j$, and let $z_j\in F_j$ be the $T$-fixed point distinct from $F_j\cap
S_i$. Endow $\O_{F_j}(z_j)$ with the $T$-linearization for which its canonical
section is invariant. We write
\[
  [z_j]_T=c_1^T(\O_{F_j}(z_j))\in A_T^1(F_j)
\]
for the equivariant Chow cohomology class of $z_j$. With the same notation for
the evaluation classes on the three moduli stacks, put
\begin{equation}
\label{Eqn: central number with cut insertions}
  \sE= \prod_{j=1}^a e_{0,j}^*[z_j]_T,\quad
  R_{\bdelta_0} =\deg\bigl(\sE\cap[\scrM_0]_T^\virt\bigr),\quad
  R_{\bdelta_0(k)}(q) =\deg\bigl((\varepsilon^*[q]\cup\sE)
                   \cap[\scrM_k]_T^\virt\bigr).
\end{equation}
We set $\sE=1$ for $a=0$. The degrees are proper equivariant pushforwards to
$A_0^T(\operatorname{pt})_\QQ=\QQ$. Since this section is nonvanishing at
$F_j\cap S_i$, the fiber there has trivial $T$-character. Thus $[z_j]_T$
restricts to zero at $F_j\cap S_i$, and $\sE$ makes a fixed contribution zero
whenever a marking labeled by $E_j$ maps into $S_i$.

Let $\scrM\to\frakM$ denote the structural morphism for any of these cases, and
let $\pi\colon\cC\to\scrM$ and $f\colon\cC^\circ\to X$ be the universal curve
and punctured map. The relative perfect obstruction theory of \cite[(4.3)--(4.4)
and Prop.\,4.2]{ACGS2} is the morphism
\[
  \phi_X\colon\EE_X=(R\pi_*f^*\Theta_X)^\vee
    \longrightarrow L_{\scrM/\frakM}.
\]

Use the line bundle $N_{Y/X}\simeq\O_Y(-D_\infty)$ and projection $p: \PP_\Xi=
N_{Y|X}\to Y$ of Proposition~\ref{Prop: Mumford degeneration}. For the toric
blowdown $b\colon Y_i\to\FF_3$ and toric positive section $H_+\subset\FF_3$, we
have $D_\infty|_{Y_i}=b^*H_+$, and hence
\begin{equation}
\label{Eqn: flap normal line}
  N_{Y|X}^\vee|_{Y_i}
    \simeq b^*\O_{\FF_3}(H_+)
    \simeq\O_{Y_i}(S_i+3F_q).
\end{equation}
In particular, $N_{Y|X}^\vee|_{Y_i}$ is nef and has degree one on each fiber
component meeting $S_i$. The section defining $D_\infty$ is $T$-invariant and
nonvanishing along $S_i$, giving an equivariant trivialization of
$N_{Y|X}|_{S_i}$ with trivial character.

The marking by one-vertex types forces every family parametrized by $\scrM$ to
factor through $D_i$, by \cite[Def.\,3.4(1)]{ACGS2}. The support argument in
\S\ref{SS: cutting at the last vertex} places each geometric fiber in $Y_i$.
Thus every punctured map $f$ factors over the open neighborhood
$U=\PP_\Xi\setminus V(\sigma-1)$ of the zero section $Y\subset\PP_\Xi$ and gives
a map $g=p\circ f$ to $Y$. The map $f$ is given over $g$ by a section of
$g^*N_{Y|X}$, the pullback along $f$ of the tautological section of
$p^*N_{Y|X}$. By \eqref{Eqn: flap normal line}, this line bundle has nonpositive
degree on every irreducible component of each geometric fiber. Its total degree
is $-3d_0$ on $\scrM_0$ and $\scrM_0^+$, and $-3d_0-k$ on $\scrM_k$. A section
vanishes on the components of negative degree, and then on all components by
connectedness. Cohomology and base change therefore give
\[
  \pi_*g^*N_{Y|X}=0,\qquad \mathbb V=R^1\pi_*g^*N_{Y|X}
    \text{ locally free}.
\]
The section defining $f$ vanishes. Hence every family factors through $Y\cap
D_i=Y_i$, also over nonreduced bases.

The logarithmic structure on $U$ is divisorial, with logarithmic boundary pulled
back from $Y$ under $p$. Thus projection to $Y$ preserves the map to $\cX$.
Applying \cite[\S4.1]{ACGS2} to the lifting problem with target $Y$ over the
same base $\frakM$ gives a perfect obstruction theory
\[
  \phi_Y\colon\EE_Y=(R\pi_*g^*\Theta_Y)^\vee
    \longrightarrow L_{\scrM/\frakM}.
\]
The tangent splitting \eqref{Eqn: zero section tangent splitting} gives
\begin{equation}
\label{Eqn: split normal obstruction theory}
  \EE_X\simeq\EE_Y\oplus\mathbb V^\vee[1],
  \qquad \phi_X=(\phi_Y,0).
\end{equation}
The normal summand maps to zero because the universal map factors through the
zero section, as follows from the cotangent and trace construction in
\cite[(4.3)--(4.4)]{ACGS2}. We thus have compatible obstruction theories
$\Phi_X,\Phi_Y$ for $\scrM\xrightarrow{\mathrm{id}}\scrM\to\frakM$, with
obstruction complex $(\mathbb V^\vee[1]\to0)$ on the identity. The virtual
pullback along the identity with this obstruction theory is multiplication by
$c_{\mathrm{top}}^T(\mathbb V)$. Manolache's functoriality theorem
\cite[Thm.\,4.8]{Manolache}, applied to the finite-dimensional approximations
defining equivariant Chow groups, therefore gives the following relation between
equivariant virtual fundamental classes:
\begin{equation}
\label{Eqn: global normal excess}
  [\scrM]_T^\virt
    =c_{\mathrm{top}}^T(\mathbb V)\cap[\scrM]_{T,\phi_Y}^\virt.
\end{equation}
Equivariant Chow localization commutes with Chern-class operators. By equivariant 
Chow localization \cite[Thm.\,5.3.5]{Kresch} and the projection
formula, it suffices for vanishing of a fixed contribution to show that
$c_{\mathrm{top}}^T(\mathbb V)$ restricts to zero on that fixed component.

\begin{lemma}
\label{Lem: negative section vanishing}
After multiplication by $\sE$, the localization contribution of any fixed
component parametrizing maps with an irreducible component of the source mapped
into $S_i$ vanishes. This holds on $\scrM_0$, on $\scrM_k$ over
$\varepsilon^{-1}(\breve S_i)$, and on $\scrM_0^+$ over
$(\varepsilon^+)^{-1}(\breve S_i)$.
\end{lemma}

\begin{proof}
Let $\scrF\subset\scrM^T$ be a fixed component, restricted to the stated
open set when there is a distinguished marking. For a punctured map $f\colon
C^\circ\to X$ over a geometric point of $\scrF$, let $C_0$ be a maximal
connected subcurve mapped into $S_i$, and let $m$ be the number of nodes joining
$C_0$ to the other components of $C$. We may assume that $C_0$ contains none of
the markings labeled by $E_j$, since otherwise $\sE$ restricts to zero. If
$f|_{C_0}$ is constant, stability gives $m\ge2$, and $m\ge3$ unless $C_0$
contains the distinguished marking. Otherwise write $f_*[C_0]=r[S_i]$, $r>0$.
The generalized balancing condition \cite[Prop.\,2.29 and Cor.\,2.30]{ACGS2} shows
\begin{equation}
\label{Eqn: endpoint contacts on negative section}
  r=D_j\cdot f_*[C_0]
    =\sum_{E\text{ attaching }C_0}
       \langle\mathbf u(E),\overline D_j\rangle,
  \qquad j\ne i.
\end{equation}
Here $\overline D_j$ is the section of the ghost sheaf of $X$ associated with
$D_j$, and $\mathbf u(E)$ is oriented away from $C_0$. The distinguished
marking, if present, contributes zero. A nonzero summand requires an attachment
over $S_i\cap D_j$. These two points are distinct, proving $m\ge2$ also in this
case.

We next construct the corresponding subcurves in the universal family over
$\scrF$. The restriction $f|_{C_0}$ is stable when the attaching points are
retained as markings, and its image is fixed pointwise. Hence $T$ acts trivially
on $C_0$. Each adjacent component maps equivariantly with positive degree onto a
fiber component meeting $S_i$, so its tangent character at the attaching point
is nonzero. At an attaching node $E$, the logarithmic curve has, locally, a
relation
\[
  s_0+s_1=\pi^\flat\rho_E
\]
in its log structure, where $s_0,s_1$ give the two branch parameters and
$\rho_E$ is a lift of the smoothing element on the base. The characters of the
branch parameters are zero and nonzero, respectively. Thus the image
$\alpha(\rho_E)$ under the structure homomorphism has nonzero character and must
vanish on $\scrF$, including over a nonreduced base. This holds at every
geometric point, so the attaching nodes form a finite \'etale family. After a
finite \'etale surjective base change choosing $C_0$ and ordering its attaching
nodes, partial normalization gives families $\cC_0$ and $\cC'$ over $\mathscr
F$. The action on $\cC_0$ is trivial, so equivariance gives a scheme-theoretic
factorization of $f|_{\cC_0}$ through $S_i$. Write $\pi_0$, $\pi'$ for the
projections to $\scrF$ and $g_0$, $g'$ for the induced maps to $Y$.

Each connected component of a geometric fiber of $\pi'$ contains an irreducible
component adjacent to $C_0$ in the original curve and hence mapping with
positive degree onto a fiber component of $a_i$. The line bundle $g'^*N_{Y|X}$
has negative degree on this irreducible component. Nonpositivity on its other
components gives $R^0\pi'_*g'^*N_{Y|X}=0$, so $R^1\pi'_*g'^*N_{Y|X}$ is locally
free. Equivariant triviality of $N_{Y|X}|_{S_i}$ and genus zero give
$\pi_{0*}g_0^*N_{Y|X}\simeq\O_{\scrF}$ and $R^1\pi_{0*}g_0^*N_{Y|X}=0$. The
normalization sequence therefore gives an exact sequence of equivariant vector
bundles
\begin{equation}
\label{Eqn: positive degree normal normalization}
  0\longrightarrow\O_{\scrF}
   \xrightarrow{\Delta}\O_{\scrF}^{\oplus m}
   \longrightarrow\mathbb V|_{\scrF}
   \longrightarrow R^1\pi'_*g'^*N_{Y|X}\longrightarrow0.
\end{equation}
Here $\Delta$ is diagonal with respect to the same trivialization at each node.
Thus $\O_{\scrF}^{\oplus m}/\Delta\O_{\scrF}$ is a trivial subbundle
of rank $m-1>0$ with trivial character and locally free quotient. It follows
that $c_{\mathrm{top}}^T(\mathbb V|_{\scrF})=0$. Pushforward along the
finite \'etale cover gives the same vanishing before base change, with rational
coefficients. Equation~\eqref{Eqn: global normal excess} proves the assertion.
\end{proof}


\subsection{The outgoing component and its contribution}
\label{SS: fixed base comparison}

We now compare the fixed contributions on $\scrM_k$ and $\scrM_0^+$. By
Lemma~\ref{Lem: negative section vanishing} and the choice of $\sE$, we may
restrict to fixed maps for which neither an irreducible component nor a marking
labeled by $E_j$ maps into $S_i$. Write $\scrF_k\subset\scrM_k^T$ for this locus
over $\breve S_i$, and $\scrF_0\subset(\scrM_0^+)^T$ for the analogous locus
with the distinguished marking $p$ mapped into $H_i$ over $\breve S_i$. Denote
the fibers under $\varepsilon$ and $\varepsilon^+$ over $q$ by $\scrF_k(q)$ and
$\scrF_0(q)$, respectively. On $\scrF_0(q)$, the marking $p$ maps to $y_q=
F_q\cap H_i$. Since $P\cdot A=0$, every map parametrized by $\scrF_0(q)$ avoids
$P$.

Let $\phi_0\colon\EE_0\to L_{\scrM_0^+/\frakM_0^+}$ be the relative perfect
obstruction theory from \S\ref{SS: relative localization}. We work one fixed
component at a time. Choose an open and closed substack
$\foF_0\subset(\frakM_0^+)^T$ containing its image, and write
$\zeta_0\colon\scrF_0\to \foF_0$ for the induced morphism. The fixed part
$\phi_0^{\fix}$ defines a virtual pullback $\zeta_0^!$. The relative virtual
localization theorem of Graber and Martin \cite[Prop.\,4.5 and
Thm.\,4.8]{GraberMartin} expresses the contribution as
\begin{equation}
\label{Eqn: closed relative localization class}
  \gamma_0
    =\frac{\zeta_0^!(\alpha_0)}
           {e_T((\EE_0^\vee)^{\mathrm{mov}})},
  \qquad
  \alpha_0=\operatorname{Res}_{\foF_0/\frakM_0^+}[\frakM_0^+]_T.
\end{equation}
Here $\operatorname{Res}$ is their residue homomorphism, and the superscripts
$\fix$ and $\mathrm{mov}$ denote the summands with trivial and nontrivial torus
characters. The denominator is $e_T(N^\virt)$ in their notation. We use rational
equivariant Chow groups, localized as in \cite[\S3.1]{GraberMartin}. Define
$\zeta_k: \scrF_k\to \foF_k$ and $\gamma_k$ analogously for
$\scrM_k\to\frakM_k$. Let $\gamma_0(q)$ and $\gamma_k(q)$ be obtained by capping
with $\sE$ and taking refined pullback along $\{q\}\hookrightarrow\breve S_i$,
using $\varepsilon^+$ and $\varepsilon$, respectively.

\begin{remark}
\label{Rem: singular localization base} The Artin stack bases can remain singular,
hence require \cite{GraberMartin}, at the images of components of the fixed
locus not discarded by Lemma~\ref{Lem: negative section vanishing}. For example,
attach four degree-one covers of $H_2$ in pairs to two contracted components,
which meet a third, all over $H_2\cap D_1$. Put a marking of contact order
$4$ on the third component and one of contact order $1$ over $H_2\cap D_3$ on
each cover. This punctured map is fixed by $T$ up to isomorphism and avoids
$S_2$. The basic-monoid calculation gives, after imposing the marking ideal, a
toric factor with monoid
\[
  \bigl\langle\ell_1,\ell_2,\lambda_1,\lambda_2
    \mid\ell_1+2\lambda_1=\ell_2+2\lambda_2\bigr\rangle^{\mathrm{sat}}.
\]
Its dual cone has dimension three and four extremal rays, so the Artin stack base
is singular at the image of this map.
\end{remark}

Let $\btau_k$ be the two-vertex refinement of $\bdelta_0(k)$ with both vertices
in the ray of $D_i$, of classes $A,kf_i$, joined by an edge $E_p$ of contact
order zero. The outgoing leg lies on the second vertex and all other legs on the
first.

\begin{lemma}
\label{Lem: outgoing fixed refinement}
The universal map over $\scrF_k$ has a canonical $\btau_k$-marking.
\end{lemma}

\begin{proof}
Let $f$ be a stable map over a geometric point of $\scrF_k(q)$, with outgoing
puncture $x$. The component $C_{\mathrm{out}}$ carrying $x$ is nonconstant,
since no component maps into $S_i$. Equivariance and the outgoing contact give a
degree-$k$ cover of $F_q$, totally ramified over $q,y_q$. Its special points are
fixed, and the other markings lie on different fibers. Hence $C_{\mathrm{out}}$
has a unique attaching node over $y_q$, since $d_0>0$. Write $C_0$ for the
remaining curve and $p,p_\infty$ for the preimages of the node on
$C_0,C_{\mathrm{out}}$.

The edge $E_p$ is contracted, so $\mathbf u(E_p)=0$. The remaining map has class
$A$ and avoids $P$, since $P\cdot A=0$ and no component maps into $P$. Thus the
component through $p$ maps into $H_i$, and stability makes the action on this
component trivial. The $T$-action on $T_pC_0$ and $T_{p_\infty}C_{\mathrm{out}}$
has characters $0$ and $-\hbar/k$, respectively.

Let $W\to\scrF_k(q)$ be a morphism. \'Etale locally on $W$, choose a section
$\tau_p$ of $\M_W$ lifting the section of $\overline{\M}_W$ associated with the
attaching node. Since $T$ acts on $\tau_p$ with nonzero weight and acts
trivially on $\underline W$, equivariance gives $\alpha_W(\tau_p)=0$. We can
therefore mark the node by $E_p$, with $\mathbf u(E_p)=0$, giving the canonical
$\btau_k$-marking.
\end{proof}

\begin{lemma}
\label{Lem: outgoing fixed splitting}
Splitting along $E_p$ and projecting to the remaining map defines a morphism
$r_k\colon\scrF_k(q)\to\scrF_0(q)$ preserving the evaluations at the markings
labeled by $E_j$. There is an isomorphism of underlying stacks
\[
  \scrF_k(q)\simeq\scrF_0(q)\times B\mu_k,
\]
under which $r_k$ is the first projection.
\end{lemma}

\begin{proof}
By Lemma~\ref{Lem: outgoing fixed refinement}, splitting the universal map over
$\scrF_k(q)$ along $E_p$ gives a remaining map marked by $\bdelta_0^+$ and a
smooth two-pointed outgoing map of class $kf_i$. The remaining map lies in
$\scrF_0(q)$, defining $r_k$, and both attaching evaluations equal $y_q$.
Splitting preserves the evaluations labeled by $E_j$.

Applying \cite[Thm.\,6.2]{Gross} to the outgoing type replaces the target $X$ by
$D_i$ with the divisorial log structure defined by the other components of
$X_0$, without changing the underlying moduli stack or obstruction theory. The
remaining boundary on $F_q$ is $\{q\}$, so the outgoing factor parametrizes
degree-$k$ covers of $(F_q,q)$ with an ordinary marking over $y_q$. After the
chosen isogeny of $T$, the outgoing fixed factor is $B\mu_k$, parametrizing
degree-$k$ covers totally ramified over $q,y_q$.

The outgoing tropical map is determined by the image of its attaching vertex.
Since $E_p$ has contact order zero, the tropical moduli cone of each glued type
is the product of that of the remaining type with the standard integral ray
recording the length of $E_p$. Thus the fs-cartesian gluing diagram of
\cite[Cor.\,5.15]{ACGS2} is cartesian on underlying stacks. Since both attaching
evaluations equal $y_q$, this identifies $\scrF_k(q)$ with
$\scrF_0(q)\times B\mu_k$, with $r_k$ the first projection.
\end{proof}

\begin{remark}
\label{Rem: exhaustion of punctured central factors}
The proofs of Lemmas~\ref{Lem: negative section vanishing} and \ref{Lem:
outgoing fixed refinement} apply to any punctured central class $A_k$ of
positive ruling degree. Each surviving fixed map carries its outgoing contact on
a degree-$k$ cover of $F_q$. Removing this component leaves an effective class
$A=A_k-kf_i$. Forgetting the ordinary attaching marking and stabilizing gives
the corresponding closed central factor, with unchanged cut data.
\end{remark}

For each actual decorated type $\eta_0$ along the image of $\scrF_0(q)$ in
$\frakM_0^+$, let $\eta_k$ be obtained by attaching the two-pointed outgoing
component along $E_p$. Let $\foU_0\subset \frakM_0^+$, $\foU_k\subset \frakM_k$
be the unions of the open neighborhoods supplied by the canonical basic-monoid
charts for these types. Their strata have types obtained by contracting
$\eta_0,\eta_k$, respectively.

After restricting to $\foU_0$, we take $\foF_0\subset \foU_0^T$ to be the locus
with tangent character zero at $p$. Similarly, let $\foF_k\subset \foU_k^T$ be
the locus where $x$ lies on a two-pointed component of class $kf_i$, with
tangent characters $0,-\hbar/k$ at the preimages $p,p_\infty$ of its attaching
node $E_p$. On $\foF_k$, let $\cL_p$ and $\cL_{p_\infty}$ denote the relative
cotangent lines at the two sections of the partial normalization. The smoothing
line is
\[
  \SS=\cL_p^\vee\otimes\cL_{p_\infty}^\vee.
\]

\begin{lemma}
\label{Lem: outgoing normal structure}
Attachment induces a $T$-equivariant isomorphism
\begin{equation}
\label{Eqn: outgoing ambient product}
  \foU_0\times[\AA^1/\GG_m]\xrightarrow{\sim}\foU_k,
\end{equation}
with trivial $T$-action on the quotient factor, inducing an isomorphism
$\foF_0\times B\GG_m\simeq \foF_k$. For the projection $\rho_k\colon \foF_k\to 
\foF_0$,
\begin{align}
\label{Eqn: outgoing normal cone product}
  C_{\foF_k/\frakM_k}
    &\simeq\rho_k^*C_{\foF_0/\frakM_0^+}\oplus\SS,\\
\label{Eqn: outgoing normal structure residue}
  \alpha_k&=e_T(\SS)^{-1}\cap\rho_k^*\alpha_0.
\end{align}
\end{lemma}

\begin{proof}
A morphism $W\to \foU_0\times[\AA^1/\GG_m]$ gives a punctured map
$(C_0/W,p,f_0)$ to $\cX$ and a line bundle with section $(L,s)$ over $W$. By
\cite[Prop.\,5.3]{Kausz}, the family of domains and $(L,s)$ classify pointed
curves $(C/W,x)$ with a contraction to $C_0$, inserting a two-pointed rational
component at $p$ over $V(s)$ and giving an isomorphism where $s$ is invertible.
The basic log structures on the stack of pointed prestable curves and its
universal curve are the divisorial ones given by the nodal boundary and its
inverse image together with the marked sections, respectively. Applied over this
stack, Kausz's construction preserves the old nodal divisors, while the local
equation $xy=s$ \cite[Lem.\,5.5]{Kausz} identifies the new one with the zero
divisor on $[\AA^1/\GG_m]$, with multiplicity one. Thus the basic base log
structure of $C/W$ is the amalgamated sum along units of that of $C_0/W$ and the
log structure associated to $(L,s)$. Pulling back the basic logarithmic curve to
the given base log structure with this additional factor gives the required
logarithmic structure on $C/W$.

Keep the tropical map unchanged away from the leg at $p$. Map $E_p$ constantly
to the image of the vertex carrying $p$. Decorate the new outgoing leg by the
cone spanned by $\check D_i$ and $\check P$, and give it contact order
$-k\be_i$. When $E_p$ contracts, use this leg at the merged vertex. The tropical
correspondence \cite[Prop.\,2.10]{ACGS1} and prestable puncturing
\cite[Prop.\,2.5]{ACGS2} extend $f_0$ to a punctured map $C/W\to \cX$. Add
$kf_i$ to the decoration of the component carrying $x$.

Contact order zero gives the canonical decomposition
\[
  Q_{\eta_k}=Q_{\eta_0}\oplus\NN.
\]
The union of the marking and puncturing ideals is extended from $Q_{\eta_0}$
\cite[Def.\,3.20 and Prop.\,3.24]{ACGS2}. Indeed, the outgoing puncturing
condition is already imposed by the marking in $D_i$, and the ambient type does
not mark $E_p$. Thus the construction respects the basic log structures and
defining ideals and gives a morphism to $\foU_k$.

Conversely, for a family in $\foU_k$, change the outgoing contact to zero and
subtract $kf_i$ from the decoration of the component carrying $x$. By the choice
of charts, the modified decorated curve is stable except possibly for a
two-pointed component of class zero carrying $x$. Stabilizing the modified
decorated curve gives a contraction $\kappa\colon C\to C_0$ over $W$, compatible
with base change. By \cite[Prop.\,5.3]{Kausz}, this contraction of pointed curves
determines a line bundle with section $(L,s)$ on $W$ and is recovered from these
data. The modified tropical map factors through the tropicalization of $C_0$, so
\cite[Prop.\,2.10]{ACGS1} gives the descended map to $\cX$. The monoid
decomposition identifies its basic monoid with the $Q_{\eta_0}$-factor, giving
the remaining family in $\foU_0$. These constructions are inverse and commute
with base change and the $T$-action.

Along $s=0$, Kausz's line bundle is the tensor product of the tangent lines at
the attaching node. For fixed objects with $L$ of character $-\hbar/k$, the
trivial action on the underlying base forces $s=0$. This proves the product
description of $\foF_k$, identifies $L|_{\foF_k}$ with $\SS$, and shows that
$\foF_0,\foF_k$ are open substacks of $(\frakM_0^+)^T, (\frakM_k)^T$,
respectively.

The product isomorphism gives \eqref{Eqn: outgoing normal cone product} and
identifies the specialized fundamental classes by flat pullback along
\[
  \rho_k^*C_{\foF_0/\frakM_0^+}\oplus\SS
    \longrightarrow C_{\foF_0/\frakM_0^+}.
\]
The formula for exact sequences of cones \cite[Expl.\,4.1.6]{Fulton} has an
equivariant residue version \cite[Lem.\,3.11]{GraberMartin}. The proof there
applies unchanged to a split sequence with a cone stack as middle term. For our
sequence of moving cone stacks with kernel $\SS$, this gives
\eqref{Eqn: outgoing normal structure residue}.
\end{proof}

\begin{proposition}
\label{Prop: outgoing contribution operator} Let
$r_k\colon\scrF_k(q)\to\scrF_0(q)$ be the morphism from Lemma~\ref{Lem:
outgoing fixed splitting} that removes the outgoing component
$C_{\mathrm{out}}$. Let $\cL_p$ be the cotangent line bundle at the marking $p$
on the universal curve over $\scrF_0(q)$, with its natural $T$-linearization,
and put $\psi_p=c_1^T(\cL_p)$. For $k>0$, the localization contributions
$\gamma_0(q),\gamma_k(q)$ introduced after \eqref{Eqn: closed relative
localization class} fulfill
\begin{equation}
\label{Eqn: outgoing contribution operator}
  r_{k*}\gamma_k(q)
    =(-1)^k\left(1+\frac{k\psi_p}{\hbar}\right)^{-1}\cap\gamma_0(q).
\end{equation}
\end{proposition}

\begin{proof}
We first compute the contribution of the degree-$k$ cover of $F_q$ to the
relative deformation complex. Let $\pi_{\mathrm{out}}\colon
\cC_{\mathrm{out}}\to\scrF_k(q)$ be the family of outgoing components,
$f_{\mathrm{out}}\colon\cC_{\mathrm{out}}^\circ\to X$ the universal outgoing
map, whose underlying map factors through $F_q$ as a degree-$k$ cover, and
$p_\infty$ its attaching section. Put
\[
  \GG_k=R\pi_{\mathrm{out}*}
    \bigl(f_{\mathrm{out}}^*\Theta_X\otimes\O(-p_\infty)\bigr).
\]
Write $\EE_0,\EE_k$ for the pull-back to $\scrF_0,\scrF_k$ of the complexes
defining the relative perfect obstruction theories of $\scrM_0^+\to\frakM_0^+$
and $\scrM_k\to\frakM_k$, respectively. The restriction map from $f^*\Theta_X$
to the remaining subcurve has kernel $f_{\mathrm{out}}^* \Theta_X
\otimes\O(-p_\infty)$, extended by zero from the outgoing component. Derived
pushforward gives the distinguished triangle
\begin{equation}
\label{Eqn: outgoing normalization triangle}
  \GG_k\longrightarrow\EE_k^\vee
    \longrightarrow r_k^*\EE_0^\vee
    \longrightarrow\GG_k[1].
\end{equation}
This triangle reduces the comparison of the relative deformation complexes to
the computation of $\GG_k$. We determine its fixed part and the inverse Euler
class of its moving part.

The projection $a_i\colon Y_i\to S_i$ and the logarithmic differential of the
degeneration parameter $\tau$ give an equivariant exact sequence for the
logarithmic tangent bundle of $Y$
\[
  0\longrightarrow T_{F_q}(-q)\longrightarrow\Theta_Y|_{F_q}
    \longrightarrow\O_{F_q}^{\oplus2}\longrightarrow0.
\]
The quotient has no cohomology after pullback to the cover and twisting by
$-p_\infty$. By \eqref{Eqn: zero section tangent splitting}, it therefore
suffices to consider $T_{F_q}(-q)\simeq\O_{F_q}(1)$ and the normal line bundle
$N_{Y|X}|_{F_q}\simeq\O_{F_q}(-1)$. The fiber $T_{F_q}(-q)|_q\simeq
T_qF_q\otimes T_q^*F_q$ has weight zero. Since $q\notin D_\infty$, the invariant
function $\sigma$ defines $Y$ near $q$ and trivializes its conormal line, so
$N_{Y|X}|_q$ also has weight zero. The tangent weight at $x$ is $\hbar/k$. Thus
$H^0 \big(f_{\mathrm{out}}^* T_{F_q}(-q)\otimes \O(-p_\infty)\big)$ has weights
$-j\hbar/k$, $0\le j<k$. Similarly, equivariant Serre duality shows that $H^1
\big(f_{\mathrm{out}}^* N_{Y|X}|_{F_q}\otimes \O(-p_\infty)\big)$ has weights
$j\hbar/k$, $1\le j\le k$. The outgoing family is pulled back from
$[\PP^1/\mu_k]\to B\mu_k$, with $\mu_k$ acting by deck transformations. Its
cohomology bundles come from $T\times\mu_k$-representations, whose
$\mu_k$-characters contribute only torsion Chern classes. The weight-zero
section is pulled back from the torus vector field on $F_q$, hence is
deck-invariant. We obtain
\begin{equation}
\label{Eqn: GG_k^mov}
  \GG_k^{\fix}\simeq\O[0],\qquad
  e_T(\GG_k^{\mathrm{mov}})^{-1}
    =\frac{\prod_{j=1}^k(j\hbar/k)}
           {\prod_{j=1}^{k-1}(-j\hbar/k)}
    =(-1)^{k-1}\hbar.
\end{equation}

We next use the fixed part of the normalization triangle to compare the virtual
pullbacks. Recall the morphisms $\zeta_0\colon\scrF_0(q)\to\foF_0$ and
$\zeta_k\colon\scrF_k(q)\to\foF_k$ induced by composition with $X\to\cX$.
Contraction of $C_{\mathrm{out}}$ commutes with this composition, so
$\rho_k\circ\zeta_k=\zeta_0\circ r_k$. Hence $\zeta_k$ factors as
\[
  \scrF_k(q)\xrightarrow{\widetilde r_k}
    \scrF_0(q)\mathbin{\times}_{\foF_0}\foF_k
    \longrightarrow \foF_k.
\]
By \eqref{Eqn: outgoing ambient product}, the intermediate stack is
$\scrF_0(q)\times B\GG_m$. After tensoring the line represented by the second
factor with $\operatorname{pr}_1^*\cL_p$, of $T$-weight zero by the definition
of $\foF_0$, the morphism $\widetilde r_k$ is the product with $B\mu_k\to
B\GG_m$ induced by the deck action on the attaching tangent line. Thus
$\widetilde r_k$ is smooth of relative dimension one.

By \cite[Lem.\,2.12]{GraberMartin}, the $T$-fixed summands of $\EE_0,\EE_k$
induce relative perfect obstruction theories for $\scrF_0\to \foF_0$ and
$\scrF_k\to \foF_k$.  Write $\zeta_0^!,\zeta_k^!$ for the induced virtual
pullbacks followed by the common refined pullback at $q$. Their compatibility
with the splitting diagram follows from \cite[(4.15), Thm.\,5.19 and
Prop.\,5.21]{ACGS2}. The evaluations to $\breve S_i$ agree under splitting, so
this compatibility is preserved by the common refined pullback at $q$. For
$\widetilde r_k$, the compatible relative obstruction theory
$(\GG_k^{\fix})^\vee\to\LL_{\widetilde r_k}$ is an isomorphism, since both
complexes are concentrated in degree zero. Its virtual pullback is therefore the
smooth pullback. Functoriality of virtual pullback \cite[Thm.\,4.8]{Manolache}
and smooth base change give
\begin{equation}
\label{Eqn: outgoing fixed virtual pullback}
  \zeta_k^!\rho_k^*\beta=r_k^*\zeta_0^!\beta,
  \qquad \beta\in A_*^T(\foF_0).
\end{equation}

By total ramification, the $k$th tensor power of the relative tangent line at
$p_\infty$ is the constant line $T_{y_q}F_q$, of weight $-\hbar$. Hence by the
definition of $\SS$ before Lemma~\ref{Lem: outgoing normal structure}, we obtain
\[
  e_T(\zeta_k^*\SS)= e_T\big( r_k^*\cL_p^\vee\otimes \cL_{p_\infty}^\vee \big)
  =-r_k^*\psi_p-\hbar/k.
\]
The residue formula \eqref{Eqn: outgoing normal structure residue}, the
compatibility of virtual pullbacks \eqref{Eqn: outgoing fixed virtual pullback},
and the normalization triangle give
\begin{alignat*}{2}
  \gamma_k(q)
    &\overset{\eqref{Eqn: closed relative localization class}}{=}
      \frac{\sE\cap\zeta_k^!(\alpha_k)}
           {e_T((\EE_k^\vee)^{\mathrm{mov}})}
    &&\overset{\substack{
        \eqref{Eqn: outgoing normal structure residue}\\
        \eqref{Eqn: outgoing fixed virtual pullback}}}{=}
      \frac{r_k^*(\sE\cap\zeta_0^!(\alpha_0))}
           {e_T((\EE_k^\vee)^{\mathrm{mov}})\,e_T(\zeta_k^*\SS)}\\
    &\overset{\substack{
		\eqref{Eqn: outgoing normalization triangle}\\
		\eqref{Eqn: closed relative localization class}}}{=}
      \frac{r_k^*\gamma_0(q)}
           {e_T(\GG_k^{\mathrm{mov}})\,e_T(\zeta_k^*\SS)}
    &&\overset{\eqref{Eqn: GG_k^mov}}{=} 
    \frac{(-1)^{k-1}\hbar}{-r_k^*\psi_p-\hbar/k} \cap r_k^*\gamma_0(q).
\end{alignat*}
Here the insertion $\sE$ pulls back because $r_k$ preserves the evaluations at
the markings labeled by $E_j$. Finally, $r_k$ is the projection from
$\scrF_0(q)\times B\mu_k$, so $r_{k*}r_k^*$ is multiplication by $1/k$. The
projection formula proves the assertion.
\end{proof}


\subsection{Polynomiality, the constant term, and the wall versus closed
comparison}
\label{SS: polynomiality and assembly}

Recall that $d_0=H\cdot A>0$, where $A=\bA(v_0)$ is the class decorating the
central vertex.

\begin{theorem}
\label{Thm: central signed polynomiality}
There is a polynomial $P_{\bdelta_0}(s)\in\QQ[s]$ of degree at most $2d_0-2$
such that
\begin{equation}
\label{Eqn: central signed polynomiality}
  P_{\bdelta_0}(k)=(-1)^kR_{\bdelta_0(k)}(q)
  \quad(k>0),\qquad
  P_{\bdelta_0}(0)=d_0R_{\bdelta_0}.
\end{equation}
\end{theorem}

\begin{proof}
Compose with $a_i\colon Y_i\to S_i$, forget all markings except $p$, and
stabilize to obtain
\[
  \nu\colon\scrF_0(q)\longrightarrow
    \overline{\scrM}_{0,1}(S_i,d_0)_q.
\]
The subscript denotes evaluation at $q$. Stabilization is an isomorphism near
$p$. Indeed, if the component carrying $p$ is nonconstant, its image lies in
$H_i$, since $f(p)=y_q$ and the $T$-fixed map avoids $P$. It therefore has
positive ruling degree. A maximal subtree contracted by the original map and
containing $p$ contains none of the markings labeled by $E_j$, and all adjacent
components have positive ruling degree, so it also survives. Thus
$\cL_p\simeq\nu^*\cL_p^{\mathrm{st}}$ equivariantly, with trivial
$T$-linearization on the right. Taking $c_1$ and noting that the codomain of
$\nu$ has dimension $2d_0-2$ gives
\begin{equation}
\label{Eqn: ruling cotangent bound}
  \psi_p^{2d_0-1}=0\qquad\text{on }\scrF_0(q).
\end{equation}

Summing over the retained fixed components, define
\begin{equation}
\label{Eqn: central polynomial coefficients}
  P_{\bdelta_0}(s)
    =\sum_{\scrF_0(q)}
      \deg\left(\left(1+\frac{s\psi_p}{\hbar}\right)^{-1}
        \cap\gamma_0(q)\right).
\end{equation}
The sum is finite, and the inverse truncates at degree $2d_0-2$. Each
coefficient has equivariant degree zero, hence lies in $\QQ$.
Proposition~\ref{Prop: outgoing contribution operator} and Lemma~\ref{Lem:
negative section vanishing} give the asserted values at positive integers.

For the constant term, the map $\mathrm{ft}\colon\scrM_0^+\to\scrM_0$ forgetting
the contact order zero marking $p$ and stabilizing is described in
\cite[Thm.\,5.5(2)--(3)]{associativity} as the saturation of the universal
punctured curve over $\scrM_0$, with compatible relative obstruction theories.
This saturation is an isomorphism away from the remaining marked sections. Since
the markings labeled by $E_j$ map to the fibers $F_j$, disjoint from $F_q$, the
image of $p$ after stabilization avoids these sections on $\ev_p^{-1}(F_q)$,
with $\ev_p$ as in \eqref{Eqn: ev_p}. Thus a neighborhood of this locus is
identified with an open substack of the universal curve over $\scrM_0$. The
usual proof of the divisor axiom therefore gives
\[
  \mathrm{ft}_*\bigl((\varepsilon^+)^*[q]\cap[\scrM_0^+]_T^\virt\bigr)
    =(F_q\cdot A)[\scrM_0]_T^\virt=d_0[\scrM_0]_T^\virt.
\]
Since $\sE$ pulls back under $\mathrm{ft}$, localization and the projection
formula yield
\[
  P_{\bdelta_0}(0) = \sum_{\scrF_0(q)}\deg\gamma_0(q)
    =\deg\bigl(((\varepsilon^+)^*[q]\cup\sE)
        \cap[\scrM_0^+]_T^\virt\bigr)
    =d_0\deg\bigl(\sE\cap[\scrM_0]_T^\virt\bigr)
     =d_0R_{\bdelta_0}.
\]
\end{proof}

\begin{corollary}
\label{Cor: wall versus closed comparison}
For every $d>0$, there are polynomials $P_{d,i}(s)\in\QQ[s]$, $i=1,2,3$, of
degree at most $2d-2$ satisfying
\begin{equation}
\label{Eqn: assembled localization comparison}
  P_{d,i}(k)=(-1)^kW_{d,-k\be_i}\quad(k>0),\qquad
  \sum_{i=1}^3P_{d,i}(0)=dN_d.
\end{equation}
Thus Assumption~\ref{Ass: wall polynomiality} holds, and the limit in
Proposition~\ref{Prop: asymptotic period} equals $-dN_d$.
\end{corollary}

\begin{proof}
In \eqref{Eqn: root numerical reduction}, a map parametrized by
$\varepsilon_0^{-1}(q)\subset\scrM(X,\bdelta_0)$ with $d(\bA(v_0))=0$ has image
in $F_q$. Since $F_q\cap F_j=\varnothing$ for all $j$, this forces $a=0$ and
hence $d=0$. Thus only $d_0>0$ contributes.

Remark~\ref{Rem: exhaustion of punctured central factors} gives a closed pair
$(\bdelta,v_0)$ for every nonzero punctured contribution, preserving the
branches and tropical parameter cone. Thus $\bdelta$ is rigid, and the
degree--contact relation and the intersection table in \S2 give total class
$d\beta$. Adding the outgoing leg and $kf_i$ reverses this operation, with
gluing data as in \eqref{Eqn: closed punctured gluing data}.

For each $i$, define
\begin{equation}
\label{Eqn: directional localization polynomials}
  P_{d,i}(s)
    =\sum_{[(\bdelta,v_0)]}
       \frac{m_\delta\mu_{\delta,v_0}}{|\Aut(\bdelta,v_0)|}
       P_{\bdelta_0}(s)\prod_{j=1}^a b_j,
\end{equation}
where the sum is that of \eqref{Eqn: vertex weighted closed splitting},
restricted to pairs with $\bsigma(v_0)$ the ray of $D_i$.

By Proposition~\ref{Prop: wall degree comparison} and~\eqref{Eqn: point
constrained delta contribution}, \eqref{Eqn: root numerical reduction},
\eqref{Eqn: closed punctured gluing data}, replacing $P_{\bdelta_0}(s)$ in this
sum by $R_{\bdelta_0(k)}(q)$ gives $W_{d,-k\be_i}$. Theorem~\ref{Thm: central
signed polynomiality} therefore gives $P_{d,i}(k)=(-1)^kW_{d,-k\be_i}$ and,
since $d_0\le d$, the degree bound.

Substituting $P_{\bdelta_0}(0)=d_0R_{\bdelta_0}$ and summing over $i$ gives
$\sum_iP_{d,i}(0)=dN_d$ by \eqref{Eqn: vertex weighted closed splitting}. The
limit follows from Proposition~\ref{Prop: asymptotic period}.
\end{proof}


\section{Finite wall formulas and tree expansions}
\label{Sec: finite wall formulas}

We first record a consequence of the polynomiality and constant-term comparison
in \eqref{Eqn: assembled localization comparison}.

\begin{theorem}
\label{Thm: finite wall formula}
For every $d>0$ we have
\begin{equation}
\label{Eqn: finite wall formula}
  N_d=-\frac1d\sum_{i=1}^3\sum_{k=1}^{2d-1}
    \binom{2d-1}{k}W_{d,-k\be_i}.
\end{equation}
\end{theorem}

\begin{proof}
Since $\deg P_{d,i}\leq2d-2$, its $(2d-1)$-st finite difference vanishes. Thus
\[
  P_{d,i}(0)=\sum_{k=1}^{2d-1}(-1)^{k-1}
    \binom{2d-1}{k}P_{d,i}(k)
  =-\sum_{k=1}^{2d-1}\binom{2d-1}{k}W_{d,-k\be_i}.
\]
Summing over $i$ and using $\sum_iP_{d,i}(0)=dN_d$ proves the assertion.
\end{proof}

For the remainder of this section we make the following assumption on the
canonical wall function of \eqref{Eqn: aggregate canonical wall function}.

\begin{assumption}
\label{Ass: explicit slab shape}
In $\QQ\lfor x,y,z,t\rfor/(xyz-t)$, we have
\[
  f=1+x+y+z+g(t),\qquad g(t)\in t\QQ\lfor t\rfor.
\]
\end{assumption}

The normalization of $\log f$ determines $g(t)$ uniquely under this assumption.
Theorem~\ref{Thm: finite wall formula} does not require the assumption.

We use the exponential tree expansion at the end of \cite[\S5]{ICM}. Consider
finite rooted trees whose leaves have labels $x,y,z$ and whose non-leaf vertices
have at least two children, without an ordering of the children. For each vertex
$v$, let $(a_v,b_v,c_v)$ count the leaves of each label in the subtree with root
vertex $v$. Require
\[
  a_v\be_1+b_v\be_2+c_v\be_3\neq0
\]
at every vertex, or equivalently that $a_v,b_v,c_v$ are not all equal. Let
$\cT_{d,k}^{\,i}$ be the set of isomorphism classes of such trees with $d+k$
leaves of the $i$-th label and $d$ of each other label, indexing $x,y,z$ by
$1,2,3$. Note $\cT_{d,k}^i$ can be interpreted as a set of isomorphism classes
of types of genus zero tropical stable maps to $\RR^2$. Write $\widehat
V_\Gamma$ for the non-leaf vertices and $\Aut(\Gamma)$ for the automorphisms
preserving the root vertex and the leaf labels. Repeated isomorphic branches are
allowed.

\begin{corollary}
\label{Cor: finite tree formula}
Under Assumption~\ref{Ass: explicit slab shape},
\begin{equation}
\label{Eqn: finite tree formula}
  N_d=-\frac1d\sum_{i=1}^3\sum_{k=1}^{2d-1}
    \frac1k\binom{2d-1}{k}
    \sum_{\Gamma\in\cT_{d,k}^{\,i}}
      \frac{(-1)^{|\widehat V_\Gamma|}}{|\Aut(\Gamma)|}.
\end{equation}
\end{corollary}

\begin{proof}
The exponential version of the ICM formula gives
\[
  kW_{d,-k\be_i}=\sum_{\Gamma\in\cT_{d,k}^{\,i}}
    \frac{(-1)^{|\widehat V_\Gamma|}}{|\Aut(\Gamma)|},
\]
since the product of the leaf labels is $t^dx_i^k$, where
$(x_1,x_2,x_3)=(x,y,z)$. Substitution in \eqref{Eqn: finite wall formula} proves
the formula. The sum is finite since there are at most $5d-1$ leaves and every
non-leaf vertex has at least two children.
\end{proof}

We next give an explicit generating formula, retaining $t$ as the canonical
parameter. Put
\[
  a_n=\frac{(-1)^{n-1}(3n-1)!}{(n!)^3},\qquad
  b_n=a_n\sum_{j=n+1}^{3n-1}\frac1j,
\]
and define series in the canonical parameter by
\begin{equation}
\label{Eqn: canonical A and B}
  A(t)=\sum_{n\geq1}a_n\frac{t^n}{(1+g(t))^{3n}},\qquad
  B(t)=\sum_{n\geq1}b_n\frac{t^n}{(1+g(t))^{3n}}.
\end{equation}
The normalization condition on $\log f$ gives
\begin{equation}
\label{Eqn: normalization and A}
  A(t)=-\log(1+g(t)).
\end{equation}
Permutation of $x,y,z$ identifies the three directional wall invariants under
Assumption~\ref{Ass: explicit slab shape}. Hence the polynomials $P_{d,i}$
agree, and we write $P_d=P_{d,i}$.

\begin{theorem}
\label{Thm: explicit closed formula}
Under Assumption~\ref{Ass: explicit slab shape},
\begin{equation}
\label{Eqn: closed A and B formula}
  \sum_{d\geq1}dN_dt^d=\frac32A(t)^2+3B(t).
\end{equation}
Equivalently,
\begin{equation}
\label{Eqn: closed expanded formula}
  \sum_{d\geq1}dN_dt^d
  =\frac32\log^2(1+g(t))+\sum_{n\geq1}\frac{(-1)^{n-1}}n\frac{(3n)!}{(n!)^3}
    \left(\sum_{j=n+1}^{3n-1}\frac1j\right)
    \frac{t^n}{(1+g(t))^{3n}}.
\end{equation}
\end{theorem}

\begin{proof}
For $n\geq1$, put
\[
  c_n(s)=\frac{(-1)^n}{(n!)^2}\prod_{j=n+1}^{3n-1}(s+j).
\]
Expanding $\log f$ and extracting the monomials $t^dx^k$ gives, for $d>0$,
\begin{equation}
\label{Eqn: signed wall extraction}
  P_d(s)=-\frac1s[t^d](1+g(t))^{-s}
    \left(\frac1s+\sum_{n\geq1}c_n(s)
      \frac{t^n}{(1+g(t))^{3n}}\right).
\end{equation}
Initially this identity holds for positive integers $s=k$, using the coefficient
$kW_{d,-k\be_1}$ in $\log f$ and the sign $(-1)^k$. For each $d$, the right-hand
side is a rational function of $s$, so equality on all positive integers proves
\eqref{Eqn: signed wall extraction} identically.

Since $c_n(0)=-a_n$ and $c_n'(0)=-b_n$, the expression in parentheses is
$s^{-1}-A(t)-sB(t)+O(s^2)$. By \eqref{Eqn: normalization and A},
$(1+g(t))^{-s}=\exp(sA(t))$. Consequently,
\[
  -\frac{e^{sA(t)}}s
    \bigl(s^{-1}-A(t)-sB(t)+O(s^2)\bigr)
  =-s^{-2}+\frac12A(t)^2+B(t)+O(s).
\]
The term $-s^{-2}$ has degree zero in $t$. Taking $s=0$ after extracting $[t^d]$
therefore gives $P_d(0)=[t^d](A(t)^2/2+B(t))$. Now $3P_d(0)=dN_d$ proves
\eqref{Eqn: closed A and B formula}, and \eqref{Eqn: normalization and A} gives
the expanded form.
\end{proof}

\begin{remark}
For comparison with Givental's formula, the noncanonical parameter is
$h=t/(1+g(t))^3$. In this notation $\log t=\log h-3\sum_{n\geq1}a_nh^n$ is the
usual mirror map, and \eqref{Eqn: closed A and B formula} is the hypergeometric
expression obtained from the local $I$-function and the mirror theorem, as in
\cite[\S10.2.1]{CoatesIritani}. The proof above derives the formula from the
wall comparison.
\end{remark}

For example, \eqref{Eqn: signed wall extraction} gives
\[
\begin{aligned}
  P_1(s)&=1,\\
  P_2(s)&=-\frac{s^2+4s+15}{4},\\
  P_3(s)&=\frac{s^4+12s^3+103s^2+348s+976}{36},\\
  P_4(s)&=-\frac{
    s^6+24s^5+370s^4+3072s^3+18961s^2+59400s+147996}{576}.
\end{aligned}
\]
Their constant terms recover $dN_d/3$. The values through degree twelve are
listed below and agree with \cite[App.\,C]{CoatesIritani}.
\begin{center}
\renewcommand{\arraystretch}{1.5}
\begin{tabular}{c r|c r|c r}
\toprule
$d$ & $N_d$ & $d$ & $N_d$ & $d$ & $N_d$ \\
\midrule
1 & $3$ & 5 & $211878/125$ & 9 & $6742982701/243$ \\
2 & $-45/8$ & 6 & $-102365/6$ & 10 & $-36001193817/100$ \\
3 & $244/9$ & 7 & $64639725/343$ & 11 & $6425982732150/1331$ \\
4 & $-12333/64$ & 8 & $-1140830253/512$ & 12 & $-9581431054999/144$ \\
\bottomrule
\end{tabular}
\end{center}

Writing $N_d=\sum_{k\mid d}n_{d/k}/k^3$, the corresponding values of
$n_1,\ldots,n_{12}$ are $3$, $-6$, $27$, $-192$, $1695$, $-17064$, $188454$,
$-2228160$, $27748899$, $-360012150$, $4827935937$, and $-66537713520$.


\sloppy

\end{document}

%% file: TwoParamMumford.tikz
\begin{tikzpicture}
  \definecolor{boundaryblue}{RGB}{20,83,148}
  \definecolor{webred}{RGB}{205,25,35}
  \definecolor{fibercyan}{RGB}{0,170,215}
  \definecolor{cellblue}{RGB}{242,246,250}

  \def\webvertices#1{%
    \coordinate (#1nw) at (-2,1);
    \coordinate (#1ne) at (1,1);
    \coordinate (#1se) at (1,-2);
  }

  \tikzset{
    x={(8mm,0)},
    y={(0,8mm)},
    line cap=round,
    line join=round,
    every node/.style={text=black,inner sep=1pt},
    diagram label/.style={font=\footnotesize},
    cell label/.style={font=\small},
    hirzebruch label/.style={
      cell label,
      fill=cellblue,
      inner sep=2pt
    },
    fiber label/.style={diagram label,text=fibercyan},
    class relations/.style={
      align=left,
      inner sep=0pt
    },
    outer web/.style={
      boundaryblue,
      line width=.8pt,
      dash pattern=on 4pt off 2.5pt
    },
    inner web/.style={webred,line width=1.25pt},
    fiber/.style={fibercyan,line width=.8pt},
    hirzebruch cell/.style={fill=cellblue}
  }
  \webvertices{s}
  \begin{scope}[shift={(1,-1)},scale=3]
    \webvertices{l}
  \end{scope}

  \coordinate (northwest) at (-2,1);
  \coordinate (northeast) at (2,2);
  \coordinate (southeast) at (1,-2);
  \coordinate (southeast long) at (2,-4);

  \coordinate (snwext) at ($(snw)+(northwest)$);
  \coordinate (sneext) at ($(sne)+(northeast)$);
  \coordinate (sseext) at ($(sse)+(southeast long)$);
  \coordinate (lnwext) at ($(lnw)+(northwest)$);
  \coordinate (lneext) at ($(lne)+(northeast)$);
  \coordinate (lseext) at ($(lse)+(southeast)$);

  \path[name path=large top] (lnw) -- (lne);
  \path[name path=large diagonal] (lnw) -- (lse);
  \path[name path=large right] (lne) -- (lse);
  \path[name path=small diagonal] (snw) -- (sse);
  \path[name path=small left ray] (snw) -- (snwext);
  \path[name path=small right ray] (sne) -- (sneext);
  \path[name path=small lower ray] (sse) -- (sseext);
  \path[name intersections={
    of=large top and small left ray,
    by=e3end
  }];
  \path[name intersections={
    of=large top and small right ray,
    by=e2start
  }];
  \path[name intersections={
    of=large diagonal and small lower ray,
    by=e1start
  }];


  \fill[hirzebruch cell]
    (snw) -- (sne) -- (e2start) -- (e3end) -- cycle;
  \fill[hirzebruch cell]
    (sne) -- (e2start) -- (lne) -- (lse) --
    (e1start) -- (sse) -- cycle;
  \fill[hirzebruch cell]
    (snw) -- (e3end) -- (lnw) -- (e1start) --
    (sse) -- cycle;

  \coordinate (box first corner) at (-7,-8);
  \coordinate (box second corner) at (7,4);
  \begin{scope}
    \clip (box first corner) rectangle (box second corner);
    \draw[outer web] (lnw) -- ($(lnw)!10!(lnwext)$);
    \draw[outer web] (lne) -- ($(lne)!10!(lneext)$);
    \draw[outer web] (lse) -- ($(lse)!10!(lseext)$);
    \draw[inner web] (snw) -- ($(snw)!10!(snwext)$);
    \draw[inner web] (sne) -- ($(sne)!10!(sneext)$);
    \draw[inner web] (sse) -- ($(sse)!10!(sseext)$);
  \end{scope}

  \draw[outer web] (lnw) -- (lne) -- (lse) -- cycle;
  \draw[inner web] (snw) -- (sne) -- (sse) -- cycle;

  \path (lnw)
    -- node[diagram label,pos=.22,below right=2pt and 7pt] {$e_3$}
    (e3end);
  \path (e2start)
    -- node[diagram label,midway,above=2.5pt] {$e_2$}
    (lne);
  \path (e1start)
    -- node[diagram label,midway,above right=1pt and -1pt] {$e_1$}
    (lse);
  \path (snw)
    -- node[diagram label,midway,below=5pt] {$\ell$}
    (sne);
  \path (sne)
    -- node[diagram label,midway,below right=1pt and 1pt] {$f$}
    (e2start);

  \path (e3end)
    -- node[diagram label,pos=.3,above=3pt] {$h_2$}
    (e2start);
  \path (lne)
    -- node[diagram label,midway,below right=11mm and 3pt] {$h_1$}
    (lse);
  \path (lnw)
    -- node[diagram label,midway,below right=15mm and 10mm] {$h_3$}
    (e1start);

  \coordinate (f2bottom) at (0,1);
  \coordinate (f2top) at (0,2);
  \draw[fiber] (f2top)
    -- node[fiber label,midway,right=4pt] {$f_2$}
    (f2bottom);

  \coordinate (f1left) at (1,-1);
  \coordinate (f1right) at (4,-1);
  \path (f1left) -- coordinate[midway] (f1mid) (f1right);
  \draw[fiber] (f1left)
    -- node[fiber label,midway,below=2pt] {$f_1$}
    (f1right);

  \coordinate (f3outer) at (-1,-2);
  \coordinate (f3inner) at (0,-1);
  \draw[fiber] (f3outer)
    -- node[fiber label,midway,above left=1pt and 1pt] {$f_3$}
    (f3inner);

  \path (snw) -- coordinate[midway] (Y2pos) (f2top);
  \path (e2start) -- coordinate[midway] (Y1upper) (f1left);
  \path (Y1upper) -- coordinate[midway] (Y1pos) (f1mid);
  \path (lnw) -- coordinate[midway] (Y3pos) (sse);
  \path (snwext) -- coordinate[midway] (D2pos) (sneext);
  \path (lne)
    -- coordinate[midway] (large right midpoint)
    (lse);
  \path (lnw)
    -- coordinate[midway] (large diagonal midpoint)
    (lse);
  \coordinate (Ppos) at
    (barycentric cs:snw=1,sne=1,sse=1);

  \node[hirzebruch label] at (Y2pos) {$Y_2$};
  \node[hirzebruch label] at (Y1pos) {$Y_1$};
  \node[hirzebruch label] at (Y3pos) {$Y_3$};
  \node[cell label, above=3mm] at (D2pos) {$D_2$};
  \node[cell label,right=10mm of large right midpoint] {$D_1$};
  \node[cell label,left=15mm of large diagonal midpoint]
    (D3label) {$D_3$};
  \node[cell label] at (Ppos) {$P$};

  \node[class relations,anchor=west] at (-6,-5.5)
    {$\begin{aligned}
       e_2 &= e_1+e_3\\
       f_1 &= f+e_2=f+e_3+e_1\\
       f_2 &= f\\
       f_3 &= f+e_3
     \end{aligned}$};

  \draw[thin] (box first corner) rectangle (box second corner);
  \pgfresetboundingbox
  \useasboundingbox (box first corner) rectangle (box second corner);
\end{tikzpicture}

%% file: MumfordInvitationWeb.tikz
\begin{tikzpicture}
  \definecolor{webred}{RGB}{190,35,45}
  \definecolor{wallblue}{RGB}{72,101,128}

  \tikzset{
    web/.style={webred,line width=.95pt},
    prism edge/.style={black!62,line width=.58pt},
    continuation/.style={black!46,line width=.52pt},
    flap edge/.style={black!50,line width=.52pt},
    complex edge/.style={black!75,line width=.62pt},
    discriminant/.style={
      webred,
      line width=.72pt,
      dash pattern=on 3pt off 2pt
    },
    projected web/.style={webred,line width=1.0pt},
    component label/.style={
      font=\small,
      inner sep=0pt,
      outer sep=0pt
    },
    graph label/.style={
      font=\footnotesize,
      inner sep=0pt,
      outer sep=0pt
    },
    projection arrow/.style={->,line width=.58pt}
  }

  \begin{scope}[
    shift={(1cm,2.05cm)},
    x={(4.00mm,.41mm)},
    y={(-.79mm,2.35mm)},
    z={(0,5.60mm)},
    line cap=round,
    line join=round
  ]
    \coordinate (a) at (-2,1,0);
    \coordinate (b) at (1,1,0);
    \coordinate (c) at (1,-2,0);

    \coordinate (ae) at (-4.8,2.4,1.4);
    \coordinate (be) at (2.4,2.4,1.4);
    \coordinate (ce) at (2.4,-4.8,1.4);

    \def\H{1.85}
    \coordinate (A) at (-2,1,\H);
    \coordinate (B) at (1,1,\H);
    \coordinate (C) at (1,-2,\H);
    \coordinate (AE) at (-4.8,2.4,{1.4+\H});
    \coordinate (BE) at (2.4,2.4,{1.4+\H});
    \coordinate (CE) at (2.4,-4.8,{1.4+\H});

    \coordinate (Aplus) at (-2,1,{\H+.55});
    \coordinate (Bplus) at (1,1,{\H+.55});
    \coordinate (Cplus) at (1,-2,{\H+.55});
    \coordinate (AEplus) at (-4.8,2.4,{1.4+\H+.55});
    \coordinate (BEplus) at (2.4,2.4,{1.4+\H+.55});
    \coordinate (CEplus) at (2.4,-4.8,{1.4+\H+.55});

    \coordinate (aefar) at (-5.65,2.825,1.825);
    \coordinate (befar) at (2.825,2.825,1.825);
    \coordinate (cefar) at (2.825,-5.65,1.825);

    \fill[wallblue,fill opacity=.16]
      (a) -- (ae) -- (AE) -- (A) -- cycle;
    \fill[wallblue,fill opacity=.16]
      (b) -- (be) -- (BE) -- (B) -- cycle;
    \fill[wallblue,fill opacity=.16]
      (c) -- (ce) -- (CE) -- (C) -- cycle;

    \draw[flap edge] (ae) -- (AE);
    \draw[flap edge] (be) -- (BE);
    \draw[flap edge] (ce) -- (CE);
    \draw[continuation] (AE) -- (AEplus);
    \draw[continuation] (BE) -- (BEplus);
    \draw[continuation] (CE) -- (CEplus);

    \draw[web] (a) -- (b) -- (c) -- cycle;
    \draw[web] (a) -- (aefar);
    \draw[web] (b) -- (befar);
    \draw[web] (c) -- (cefar);

    \fill[wallblue!24,fill opacity=.55]
      (a) -- (b) -- (B) -- (A) -- cycle;
    \fill[wallblue!46,fill opacity=.74]
      (b) -- (c) -- (C) -- (B) -- cycle;
    \fill[wallblue!40,fill opacity=.70]
      (c) -- (a) -- (A) -- (C) -- cycle;
    \fill[wallblue!58,fill opacity=.62]
      (a) -- (b) -- (c) -- cycle;

    \draw[prism edge] (a) -- (A);
    \draw[prism edge] (b) -- (B);
    \draw[prism edge] (c) -- (C);
    \draw[continuation] (A) -- (Aplus);
    \draw[continuation] (B) -- (Bplus);
    \draw[continuation] (C) -- (Cplus);

    \node[component label] at (0,0,.72) {$P$};
    \node[component label] at (-1.0,3.65,1.95) {$D_2$};
    \node[component label] at (3.0,-.1,1.95) {$D_1$};
    \node[component label] at (-2.65,-1.8,0.5) {$D_3$};
  \end{scope}

  \begin{scope}[
    shift={(6.55cm,3.35cm)},
    x={(4.45mm,.45mm)},
    y={(-.89mm,2.67mm)},
    z={(0,8.1mm)},
    line cap=round,
    line join=round
  ]
    \coordinate (at) at (-2,1,.75);
    \coordinate (bt) at (1,1,.75);
    \coordinate (ct) at (1,-2,.75);
    \coordinate (ab) at (-2,1,-.75);
    \coordinate (bb) at (1,1,-.75);
    \coordinate (cb) at (1,-2,-.75);

    \coordinate (am) at (-2,1,0);
    \coordinate (bm) at (1,1,0);
    \coordinate (cm) at (1,-2,0);

    \coordinate (atext) at (-3.8,1.9,1.65);
    \coordinate (btext) at (1.9,1.9,1.65);
    \coordinate (ctext) at (1.9,-3.8,1.65);
    \coordinate (abext) at (-3.8,1.9,-1.65);
    \coordinate (bbext) at (1.9,1.9,-1.65);
    \coordinate (cbext) at (1.9,-3.8,-1.65);

    \coordinate (amext) at (-4.4,2.2,0);
    \coordinate (bmext) at (2.2,2.2,0);
    \coordinate (cmext) at (2.2,-4.4,0);

    \fill[black!3] (ab) -- (bb) -- (cb) -- cycle;
    \fill[black!4] (ab) -- (bb) -- (bt) -- (at) -- cycle;
    \fill[black!5] (bb) -- (cb) -- (ct) -- (bt) -- cycle;
    \fill[black!2] (cb) -- (ab) -- (at) -- (ct) -- cycle;
    \fill[black!3] (at) -- (bt) -- (ct) -- cycle;

    \draw[complex edge] (at) -- (bt) -- (ct) -- cycle;
    \draw[complex edge] (ab) -- (bb) -- (cb) -- cycle;
    \draw[complex edge] (at) -- (ab);
    \draw[complex edge] (bt) -- (bb);
    \draw[complex edge] (ct) -- (cb);
    \draw[complex edge] (at) -- (atext);
    \draw[complex edge] (bt) -- (btext);
    \draw[complex edge] (ct) -- (ctext);
    \draw[complex edge] (ab) -- (abext);
    \draw[complex edge] (bb) -- (bbext);
    \draw[complex edge] (cb) -- (cbext);

    \draw[discriminant] (am) -- (bm) -- (cm) -- cycle;
    \draw[discriminant] (am) -- (amext);
    \draw[discriminant] (bm) -- (bmext);
    \draw[discriminant] (cm) -- (cmext);

    \node[component label] at (0,0,0) {$P$};
    \node[component label] at (-.5,0,1.6) {$P_+$};
    \node[component label] at (0,0,-1.85) {$P_-$};
    \node[component label] at (-1.5,6,0) {$D_2$};
    \node[component label] at (3.05,-.25,0) {$D_1$};
    \node[component label] at (-2.8,-3,0) {$D_3$};
  \end{scope}

  \draw[projection arrow] (8.87,3) -- (9.60,3);

  \begin{scope}[
    shift={(11cm,3cm)},
    x={(3.60mm,0)},
    y={(0,3.60mm)},
    line cap=round,
    line join=round
  ]
    \coordinate (a) at (-2,1);
    \coordinate (b) at (1,1);
    \coordinate (c) at (1,-2);
    \coordinate (aext) at (-4.4,2.2);
    \coordinate (bext) at (2.2,2.2);
    \coordinate (cext) at (2.2,-4.4);

    \draw[projected web] (a) -- (b) -- (c) -- cycle;
    \draw[projected web] (a) -- (aext);
    \draw[projected web] (b) -- (bext);
    \draw[projected web] (c) -- (cext);

    \node[component label,font=\footnotesize] at (0,0) {$P$};
    \node[component label,font=\footnotesize] at (-.2,2.45) {$D_2$};
    \node[component label,font=\footnotesize] at (2.6,-.35) {$D_1$};
    \node[component label,font=\footnotesize] at (-1.8,-1.8) {$D_3$};
  \end{scope}
\end{tikzpicture}

%% file: P2local.tikz
\begin{tikzpicture}[
  line cap=round,
  line join=round,
  edge/.style={draw=black,line width=.4pt},
  outeredge/.style={draw=black,line width=.75pt},
  backedge/.style={draw=black!50,line width=.75pt},
  cell/.style={fill=tropicalblue,fill opacity=.18},
  planeedge/.style={draw=planecolor,line width=.4pt},
  planeouter/.style={draw=planecolor,line width=.75pt},
  planeback/.style={draw=planecolor!65,line width=.6pt},
  planecell/.style={fill=planecolor,fill opacity=.18}
]
  \definecolor{tropicalblue}{RGB}{72,101,128}
  \definecolor{planecolor}{RGB}{165,95,70}

  \begin{scope}[
    shift={(1.25cm,1.20cm)},
    x={(1.00cm,0cm)},
    y={(0cm,1.00cm)}
  ]
    \coordinate (sigma-centre) at (0,0);
    \coordinate (sigma-one) at (1,0);
    \coordinate (sigma-two) at (0,1);
    \coordinate (sigma-three) at (-1,-1);

    \fill[planecell]
      (sigma-centre) -- (sigma-one) -- (sigma-two) -- cycle;
    \fill[planecell]
      (sigma-centre) -- (sigma-two) -- (sigma-three) -- cycle;
    \fill[planecell]
      (sigma-centre) -- (sigma-three) -- (sigma-one) -- cycle;
    \draw[planeouter]
      (sigma-one) -- (sigma-two) -- (sigma-three) -- cycle;
    \draw[planeedge] (sigma-centre) -- (sigma-one);
    \draw[planeedge] (sigma-centre) -- (sigma-two);
    \draw[planeedge] (sigma-centre) -- (sigma-three);
  \end{scope}

  \node[scale=1.25] at (3.08cm,1.20cm) {$\lhook\joinrel\longrightarrow$};

  \begin{scope}[
    shift={(4.95cm,1.10cm)},
    x={(-1.14cm,.08cm)},
    y={(.55cm,.43cm)},
    z={(0cm,.95cm)}
  ]
    \coordinate (p) at (0,0,0);
    \coordinate (d-one) at (1,0,0);
    \coordinate (d-two) at (0,1,0);
    \coordinate (d-three) at (-1,-1,0);
    \coordinate (p-plus) at (0,0,1);
    \coordinate (p-minus) at (0,0,-1);

    \fill[planecell] (p) -- (d-one) -- (d-two) -- cycle;
    \fill[planecell] (p) -- (d-two) -- (d-three) -- cycle;
    \fill[planecell] (p) -- (d-three) -- (d-one) -- cycle;
    \fill[cell] (p) -- (p-plus) -- (d-one) -- cycle;
    \fill[cell] (p) -- (p-plus) -- (d-two) -- cycle;
    \fill[cell] (p) -- (p-plus) -- (d-three) -- cycle;
    \fill[cell] (p) -- (p-minus) -- (d-one) -- cycle;
    \fill[cell] (p) -- (p-minus) -- (d-two) -- cycle;
    \fill[cell] (p) -- (p-minus) -- (d-three) -- cycle;

    \draw[backedge] (d-one) -- (d-two);
    \draw[outeredge] (d-two) -- (d-three) -- (d-one);
    \draw[planeedge] (p) -- (d-one);
    \draw[planeedge] (p) -- (d-two);
    \draw[planeedge] (p) -- (d-three);
    \draw[outeredge] (p-plus) -- (d-one);
    \draw[outeredge] (p-plus) -- (d-two);
    \draw[outeredge] (p-plus) -- (d-three);
    \draw[outeredge] (p-minus) -- (d-one);
    \draw[backedge] (p-minus) -- (d-two);
    \draw[outeredge] (p-minus) -- (d-three);
    \draw[edge] (p-plus) -- (p) -- (p-minus);
  \end{scope}

\end{tikzpicture}

%% file: TypesExamples.tikz
\begin{tikzpicture}[
  line cap=round,
  line join=round
]
  \definecolor{webred}{RGB}{205,25,35}
  \definecolor{tropicalblue}{RGB}{0,50,255}

  \tikzset{
    inner edge/.style={webred,line width=1.05pt},
    outer edge/.style={black!70,line width=.68pt,
      dash pattern=on 3.2pt off 2.1pt},
    lower edge/.style={black!78,line width=.55pt},
    type curve/.style={tropicalblue,line width=1.05pt},
    marked point/.style={fill=tropicalblue,draw=none},
    green point/.style={fill=tropicalblue,draw=none},
    class label/.style={font=\tiny,text=black,inner sep=.6pt},
    flow arrow/.style={tropicalblue,line width=1.05pt,->,>=stealth}
  }

  \newcommand{\TypesUpperBase}[1]{%
    \coordinate (#1-snw) at (-2,1);
    \coordinate (#1-sne) at (1,1);
    \coordinate (#1-sse) at (1,-2);
    \coordinate (#1-lnw) at (-5,2);
    \coordinate (#1-lne) at (4,2);
    \coordinate (#1-lse) at (4,-7);

    \coordinate (#1-e-three-end) at (-4,2);
    \coordinate (#1-e-two-start) at (2,2);
    \coordinate (#1-e-one-start) at (3,-6);

    \coordinate (#1-f-two-top) at (0,2);
    \coordinate (#1-f-two-bottom) at (0,1);
    \coordinate (#1-f-one-left) at (1,-1);
    \coordinate (#1-f-one-right) at (4,-1);
    \coordinate (#1-f-three-outer) at (-1,-2);
    \coordinate (#1-f-three-inner) at (0,-1);

    \draw[outer edge]
      (-7,3) -- (#1-lnw) -- (#1-lne) -- (6,4);
    \draw[outer edge]
      (#1-lnw) -- (#1-lse) -- (4.5,-8);
    \draw[outer edge] (#1-lne) -- (#1-lse);

    \draw[inner edge]
      (-7,3.5) -- (#1-snw) -- (#1-sne) -- (4,4);
    \draw[inner edge]
      (#1-snw) -- (#1-sse) -- (4,-8);
    \draw[inner edge] (#1-sne) -- (#1-sse);
  }

  \newcommand{\TypesLowerBase}[1]{%
    \coordinate (#1-centre) at (0,0);
    \coordinate (#1-one) at (1,0);
    \coordinate (#1-two) at (0,1);
    \coordinate (#1-three) at (-1,-1);

    \draw[lower edge]
      (#1-one) -- (#1-two) -- (#1-three) -- cycle;
    \draw[lower edge] (#1-centre) -- (#1-one);
    \draw[lower edge] (#1-centre) -- (#1-two);
    \draw[lower edge] (#1-centre) -- (#1-three);
  }

  \begin{scope}[shift={(2.05,4.40)},x=.265cm,y=.265cm]
    \TypesUpperBase{a}

    \coordinate (a-top-cusp) at (-.30,1.95);
    \coordinate (a-right-cusp) at (1.95,-.30);
    \coordinate (a-left-cusp) at (-.82,-.82);

    \coordinate (a-tr-one) at (-.15,.3);
    \coordinate (a-tr-two) at (.24,-.30);
    \coordinate (a-rl-one) at (.24,-.30);
    \coordinate (a-rl-two) at (-.22,-.22);
    \coordinate (a-lt-one) at (-.22,-.22);
    \coordinate (a-lt-two) at (-.15,.3);

    \draw[type curve] (-.18,1) -- (-.18,2);
    \draw[type curve] (1,-.25) -- (4,-.25);
    \draw[type curve] (-.5,-.5) -- (-1.5,-1.5);

    \fill[tropicalblue]
      (a-top-cusp)
      .. controls (a-tr-one) and (a-tr-two)
      .. (a-right-cusp)
      .. controls (a-rl-one) and (a-rl-two)
      .. (a-left-cusp)
      .. controls (a-lt-one) and (a-lt-two)
      .. cycle;

    \foreach \p in {(-.18,1),(1,-.25),(-.5,-.5)}
      \fill[marked point] \p circle[radius=1.3pt];
  \end{scope}

  \begin{scope}[shift={(1.45,1.50)},x=1.18cm,y=1.18cm]
    \TypesLowerBase{a-low}
    \draw[type curve]
      (a-low-three) -- (a-low-centre) -- (a-low-one);
    \draw[type curve] (a-low-centre) -- (a-low-two);
    \foreach \p in {a-low-one,a-low-two,a-low-three}
      \fill[green point]
        ($(\p)!.02!(a-low-centre)$) circle[radius=1.8pt];
    \fill[green point] (a-low-centre) circle[radius=1.8pt];
    \node[class label,above=2pt] at (a-low-two) {$f_2$};
    \node[class label,below right=2pt and 2pt] at (a-low-centre) {$\ell$};
    \node[class label,right=3pt] at (a-low-one) {$f_1$};
    \node[class label,above left=2pt and 1pt] at (a-low-three) {$f_3$};
  \end{scope}

  \begin{scope}[shift={(5.85,4.40)},x=.265cm,y=.265cm]
    \TypesUpperBase{b}

    \draw[type curve]
      (b-e-two-start) -- (b-sne) -- (b-sse);
    \draw[flow arrow]
      ($(b-sse)+(0,-.10)$) -- +(0,-1.15);
    \foreach \p in {b-sne,b-sse}
      \fill[marked point] (\p) circle[radius=1.3pt];
  \end{scope}

  \begin{scope}[shift={(5.25,1.50)},x=1.18cm,y=1.18cm]
    \TypesLowerBase{b-low}
    \coordinate (b-low-upper) at (.4,.6);
    \coordinate (b-low-mid) at (.4,0);
    \draw[type curve] (b-low-upper) -- (b-low-mid);
    \fill[green point] (b-low-upper) circle[radius=1.8pt];
    \fill[green point] (b-low-mid) circle[radius=1.8pt];
    \draw[flow arrow] (b-low-mid) -- +(0,-.3);
    \node[class label,above right=2pt and 2pt] at (b-low-upper) {$f_2$};
    \node[class label,above right=2pt and 2pt] at (b-low-mid) {$\ell$};
  \end{scope}

  \begin{scope}[shift={(9.65,4.40)},x=.265cm,y=.265cm]
    \TypesUpperBase{c}
    \coordinate (c-mark) at ($(c-sne)!.43!(c-sse)$);

    \draw[type curve]
      (c-e-two-start) -- (c-sne) -- (c-sse) --
      (c-e-one-start) -- (c-lse);
    \draw[flow arrow] (c-sne) -- +(-.78,-.78);
    \foreach \p in {c-sne,c-sse}
      \fill[marked point] (\p) circle[radius=1.3pt];
  \end{scope}

  \begin{scope}[shift={(9.05,1.50)},x=1.18cm,y=1.18cm]
    \TypesLowerBase{c-low}

    \coordinate (c-low-upper-edge) at (.32,.68);
    \coordinate (c-low-interior) at (.52,0);
    \coordinate (c-low-new) at (.32,.20);
    \coordinate (c-low-lower) at (.68,-.16);
    \coordinate (c-low-arrow-end) at
      ($(c-low-centre)!.12!(c-low-one)$);

    \draw[type curve]
      (c-low-upper-edge) -- (c-low-new) -- (c-low-interior) --
      (c-low-lower) -- (c-low-one);
    \foreach \p in
      {c-low-upper-edge,c-low-new,c-low-interior,c-low-lower}
      \fill[green point] (\p) circle[radius=1.8pt];
    \fill[green point]
      ($(c-low-one)!.02!(c-low-centre)$) circle[radius=1.8pt];
    \draw[flow arrow] (c-low-new) -- (c-low-arrow-end);
    \node[class label,above right=2pt and 2pt]
      at (c-low-upper-edge) {$2f_2$};
    \node[class label,left=3pt,yshift=2pt] at (c-low-new) {$0$};
    \node[class label,above right=2pt and 0pt] at (c-low-interior) {$\ell$};
    \node[class label,right=3pt] at (c-low-one) {$2e_1$};
    \node[class label,below=3pt] at (c-low-lower) {$3f_3$};
  \end{scope}

  \begin{scope}[shift={(13.45,4.40)},x=.265cm,y=.265cm]
    \TypesUpperBase{d}

    \draw[type curve] (d-lnw) -- (d-lne);
    \draw[type curve] (d-f-two-top) -- (d-f-two-bottom);
    \draw[flow arrow]
      ($(d-f-two-bottom)+(0,-.08)$) -- +(0,-1.05);
    \foreach \p in
      {d-e-three-end,d-e-two-start,d-f-two-top,d-f-two-bottom}
      \fill[marked point] (\p) circle[radius=1.3pt];
  \end{scope}

  \begin{scope}[shift={(12.85,1.50)},x=1.18cm,y=1.18cm]
    \TypesLowerBase{d-low}
    \draw[type curve]
      (d-low-three) -- (d-low-two) -- (d-low-one);
    \foreach \p in {d-low-one,d-low-two,d-low-three}
      \fill[green point]
        ($(\p)!.02!(d-low-centre)$) circle[radius=1.8pt];
    \draw[flow arrow] (d-low-two) -- (d-low-centre);
    \node[class label,above=2pt] at (d-low-two) {$h_2+f_2$};
    \node[class label,right=3pt] at (d-low-one) {$e_2$};
    \node[class label,above left=2pt and 1pt] at (d-low-three) {$e_3$};
  \end{scope}

  \pgfresetboundingbox
  \useasboundingbox (-.02,.25) rectangle (15.20,5.50);
\end{tikzpicture}